\documentclass[onefignum,onetabnum]{siamart190516}

\usepackage{lipsum}
\usepackage{amsfonts}
\usepackage{graphicx}
\usepackage{epstopdf}
\usepackage{algorithmic}
\ifpdf
\DeclareGraphicsExtensions{.eps,.pdf,.png,.jpg}
\else
\DeclareGraphicsExtensions{.eps}
\fi

\newsiamremark{remark}{Remark}
\newsiamremark{hypothesis}{Hypothesis}
\crefname{hypothesis}{Hypothesis}{Hypotheses}
\newsiamthm{claim}{Claim}

\headers{Stochastic optimization over proximally smooth sets}{D. Davis, D. Drusvyatskiy, and  Z. Shi}

\title{Stochastic optimization over \\ proximally smooth sets}

\author{Damek Davis\thanks{School of ORIE, Cornell University,
		Ithaca, NY 14850, USA
		(\email{dsd95@cornell.edu}, \url{people.orie.cornell.edu/dsd95/}).}
	\and Dmitriy Drusvyatskiy\thanks{Department of Mathematics, U. Washington,
		Seattle, WA 98195
		(\email{ddrusv@uw.edu}, \url{sites.google.com/uw.edu/ddrusv}). Research was supported by the NSF DMS 1651851 and CCF 1740551 awards.}
	\and Zhan Shi\thanks{Department of Mathematics, U. Washington,
		Seattle, WA 98195
		(\email{zhansh@uw.edu})}
}
\usepackage{amsopn}

\ifpdf
\hypersetup{
  pdftitle={Stochastic optimization over proximally smooth sets},
  pdfauthor={Damek Davis, Dmitriy Drusvyatskiy, and  Zhan Shi}
}
\fi

\usepackage{amssymb}
\usepackage{amsmath}
\usepackage{enumerate}
\usepackage{hyperref}
\usepackage{cite}
\usepackage{color}
\usepackage{array}
\usepackage{pgfplots}

\usepackage{tikz}
\usetikzlibrary{positioning}
\usetikzlibrary{shadows}
\tikzset{
	invisible/.style={opacity=0},
	alt/.code args={<#1>#2#3}{\alt<#1>{\pgfkeysalso{#2}}{\pgfkeysalso{#3}}},
	visible on/.style={alt=#1{}{invisible}},
}
\usetikzlibrary{shapes,arrows,positioning,calc}
\usetikzlibrary{fit}\usetikzlibrary{arrows,shapes,positioning}
\usetikzlibrary{decorations.markings}
\tikzstyle arrowstyle=[scale=1]
\tikzstyle directed=[postaction={decorate,decoration={markings,
		mark=at position .5 with {\arrow[arrowstyle]{stealth}}}}]
\tikzstyle reverse directed=[postaction={decorate,decoration={markings,
		mark=at position .5 with {\arrowreversed[arrowstyle]{stealth};}}}]

\usepackage{subcaption}

\newtheorem{assumption}{Assumption}

\numberwithin{equation}{section}
\newcommand{\inclu}[0] {\ar@{^{(}->}}

\newcommand{\E}{\mathbb{E}}

\newcommand{\osa}{\mu}
\newcommand{\accD}{\tau_1}
\newcommand{\accR}{\tau_2}

\newcommand{\dist}{{\rm dist}}
\newcommand{\R}{\mathbb{R}}

\newcommand{\cR}{\mathcal{R}}
\newcommand{\cM}{\mathcal{M}}
\newcommand{\EE}{\mathbb{E}}

\newcommand{\RR}{\mathbb{R}}
\newcommand{\BB}{\mathbb{B}}

\newcommand{\cX}{\mathcal{X}}

\newcommand{\cC}{\mathcal{C}}
\newcommand{\cP}{\mathcal{P}}

\newcommand{\proj}{\mathrm{proj}}

\newcommand{\dom}{{\rm dom}\,}
\newcommand{\argmin}{\operatornamewithlimits{argmin}}

\usepackage{mathtools}
\DeclarePairedDelimiter{\dotp}{\langle}{\rangle}

\usepackage{graphicx}
\graphicspath{{./figures/}}

\begin{document}

\maketitle

% REQUIRED
\begin{abstract}
  We introduce a class of stochastic algorithms for minimizing weakly convex functions over proximally smooth sets. As their main building blocks, the algorithms use simplified models of the objective function and the constraint set, along with a retraction operation to restore feasibility. All the proposed methods come equipped with a  finite time efficiency guarantee in terms of a natural stationarity measure. We discuss consequences for nonsmooth optimization over smooth manifolds and over sets cut out by weakly-convex inequalities.
\end{abstract}

% REQUIRED

% REQUIRED

\section{Introduction}

Stochastic optimization methods
play 
a central role 
in 
statistical and machine learning. 
Departing from the classical convex setting,
nonconvexity and nonsmoothness
feature   
in contemporary applications, 
motivating
		new work
			on
			algorithmic 
				foundations
					and 
				complexity theory~\cite{doi:10.1137/17M1151031,davis2019stochastic,davis2019stochasticgeom,asi1,asi2,duchi2018stochastic}. 
A favorable problem class  
highlighted 
by this line of work
	consists of
\emph{weakly convex} function minimization over convex constraint sets.
This class 
of
functions
is 
broad, allowing for nonsmooth and nonconvex objectives, 
and in particular
includes
all compositions 
of 
Lipschitz convex functions 
with 
smooth nonlinear maps.
While this line of work allows for weakly convex objectives, 
what
is missing is
an analogous framework for nonconvex constraints sets. 
%one 
%for designing and analyzing algorithms 
%	that 
%		both  
%			enjoy similar complexity guarantees
%				and
%			support nonconvex constraint sets. 
This is an important issue: 
one 
is often interested in 
optimizing nonsmooth functions---even those that are convex---over nonconvex sets, 
for example, 
over 
embedded submanifolds 
or over sets cut out by
nonconvex functional constraints~\cite{absil2019collection,10.1145/3362077.3362085}. 
Nonconvex constraint sets,
	however,
	introduce
new complications into the implementation and analysis
	of
	stochastic optimization algorithms:
		they seemingly
			require 
		us 
			to globally solve
		a series of nonconvex constrained optimization problems
			and
		they
			appear 
		to prevent
		the use of the standard tools for understanding complexity of optimization algorithms. 
In this work,
we overcome these issues and develop numerical methods for minimizing weakly convex functions over the class of \emph{proximally smooth sets}  \cite{clarkelower}---a broad class that 
includes  
closed convex sets, sublevel sets of weakly convex functions,
and
compact $C^2$-submanifolds of $\RR^d.$

Setting the stage, 
consider 
the stochastic optimization problem
\begin{equation}\label{eqn:opt_prob}
\min_x~ \EE_{\xi \sim P} \left[f(x,\xi)\right]\quad \text{subject to }x\in \cX,
\end{equation}
where 
the data $\xi$ 
follows 
a fixed but unknown probability distribution $P$,
the set $\cX\subseteq\R^d$ 
is 
closed, 
and 
the loss functions $f(\cdot,\xi)$ 
are 
$\rho$-weakly convex, 
meaning 
that the assignment $x\mapsto f(x,\xi)+\frac{\rho}{2}\|x\|^2$ 
is 
convex. 
Our approach to this problem
will 
draw on
and
extend 
the \emph{stochastic model-based algorithm of}~\cite{davis2019stochastic}, 
which proceeds as follows.
At each iteration $t$, the algorithm
samples
$\xi_t \sim P$
and then 
minimizes
a quadratic perturbation 
of
a simplified \emph{model} $f_{x_t}(\cdot, \xi_t)$ 
of the loss
$f(\cdot, \xi_t)$ 
over 
the constraint set: 
\begin{equation}\label{eqn:basic_model_alg}
x_{t+1} = \argmin_{x \in \cX}~ \left\{f_{x_t}(x,\xi_t) + \frac{\beta_t}{2} \|x - x_t\|^2\right\}.
\end{equation}
Here $\beta_t>0$ 
is a user-specified
control sequence. 
The paper~\cite{davis2019stochastic} 
develops
rigorous efficiency guarantees
for 
this algorithm 
under 
the following assumptions:
(a) the set $\cX$ is convex
and 
(b) 
the models $f_{x_t}(\cdot, \xi)$ 
are 
weakly convex 
and 
lower-bound $f$ up to a quadratic error in expectation. 
Such models, and variants, 
have been investigated 
at great length 
in 
the  papers \cite{davis2019stochastic,davis2019stochasticgeom,asi1,asi2}
and 
we refer 
the reader 
to 
these works 
for 
detailed examples. 
With these assumptions,
the method~\eqref{eqn:basic_model_alg}  
subsumes a number of important algorithms, 
such as the
stochastic 
subgradient, 
prox-linear, 
and 
prox-point 
methods, along with their clipped variants introduced in \cite{asi2}. See Table~\ref{table:functionmodels} and Figure~\ref{fig:illustr_lower_model} for an illustration.
\begin{table}[t!]
	\centering
	\begin{tabular}{|l|l|l|}
		\hline
		Algorithm & Objective $f(x)$ & Model $f_x(y,z)$ \\
		\hline
		Proximal point & $\EE_{\xi} f(x,\xi)$ & $f(y,\xi)$ \\
		Subgradient  & $\EE_{\xi} f(x,\xi)$ & $f(x,\xi)+\langle G(x,\xi),y-x\rangle$  \\
		Clipped subgradient & $\EE_{\xi} f(x,\xi)$ & $\max\{f(x,\xi)+\langle G(x,\xi),y-x\rangle,0\}$   \\
		Prox-linear & $\EE_{z} h(c(x,\xi),\xi)$ & $h(c(x,\xi)+\nabla c(x,\xi)(y-x),\xi)$\\
		\hline
	\end{tabular}
	\caption{Four typical models; the loss function $f(\cdot,\xi)$ is assumed to be weakly convex, the stochastic subgradient $G$ satisfies $\EE_\xi[G(x, \xi)] \in \partial f(x)$, the outer loss $h(\cdot,\xi)$ is Lipschitz and convex and $c(\cdot,\xi)$ is $C^1$-smooth with Lipschitz Jacobian.}\label{table:functionmodels}
\end{table}

	\begin{figure}[t!]
	\begin{subfigure}{.25\textwidth}
		\begin{center}
			\begin{tikzpicture}[scale=0.45]
			\pgfplotsset{every tick label/.append style={font=\large}}
			\begin{axis}[%
			domain = -2:2,
			samples = 200,
			axis x line = center,
			axis y line = center,
			xtick={0.5},
			ytick=\empty
			%        ticks = none
			]		
			\addplot[black, ultra thick] {abs(x^2-1)} [yshift=3pt];
			{         \addplot[purple, very thick] {abs(x^2 - 1)} [yshift=6pt] ;      }
			%		       {            \addplot[purple, very thick] [domain=-1:2]{abs(0.75-(x-0.5))+1.2*(x-0.5)^2} [yshift=3pt] node[pos=.95,left] {\footnotesize{$f_x+(x-0.5)^2$}};    }         
			%           \only<3>       {            \addplot[violet,dashed, very thick] [domain=-1.5:1.7]{abs(0.75-(x-0.5))-1.2*(x-0.5)^2} [yshift=3pt] node[pos=.95,left] {\footnotesize{$f_x-(x-0.5)^2$}};    }        
			%      \draw[fill] (1,1) circle [radius=5];
			
			\addplot [only marks,mark=*] coordinates { (0.5,0.75) };
			%		{   \addplot [only marks, mark=*, mark size=1.5] coordinates { (11/12,0) };}
			\end{axis}		
			\end{tikzpicture}
		\end{center}
		%	\caption{$f(x)=|x^2-1|$}
	\end{subfigure}%
	\begin{subfigure}{.25\textwidth}
		\begin{center}
			\begin{tikzpicture}[scale=0.45]
			\pgfplotsset{every tick label/.append style={font=\large}}
			\begin{axis}[%
			domain = -2:2,
			samples = 200,
			axis x line = center,
			axis y line = center,
			xtick={0.5},
			ytick=\empty,
			ymin=-0.2
			%        ticks = none
			]		
			\addplot[black, ultra thick] {abs(x^2-1)} [yshift=3pt] ;
			{         \addplot[purple, very thick] {abs(.5^2 - 1) - (2*(.5))*(x - .5)} [yshift=0pt] ;      }
			%		       {            \addplot[purple, very thick] [domain=-1:2]{abs(0.75-(x-0.5))+1.2*(x-0.5)^2} [yshift=3pt] node[pos=.95,left] {\footnotesize{$f_x+(x-0.5)^2$}};    }         
			%           \only<3>       {            \addplot[violet,dashed, very thick] [domain=-1.5:1.7]{abs(0.75-(x-0.5))-1.2*(x-0.5)^2} [yshift=3pt] node[pos=.95,left] {\footnotesize{$f_x-(x-0.5)^2$}};    }        
			%      \draw[fill] (1,1) circle [radius=5];
			
			\addplot [only marks,mark=*] coordinates { (0.5,0.75) };
			%		{   \addplot [only marks, mark=*, mark size=1.5] coordinates { (11/12,0) };}
			\end{axis}		
			\end{tikzpicture}
		\end{center}
	\end{subfigure}%
	\begin{subfigure}{.25\textwidth}
		\begin{center}
			\begin{tikzpicture}[scale=0.45]
			\pgfplotsset{every tick label/.append style={font=\large}}
			\begin{axis}[%
			domain = -2:2,
			samples = 200,
			axis x line = center,
			axis y line = center,
			xtick={0.5},
			ytick=\empty
			%        ticks = none
			]		
			\addplot[black, ultra thick] {abs(x^2-1)} [yshift=3pt];
			{         \addplot[purple, very thick] {max(abs(.5^2 - 1) - (2*(.5))*(x - .5), 0)} [yshift=10pt] ;      }
			%		       {            \addplot[purple, very thick] [domain=-1:2]{abs(0.75-(x-0.5))+1.2*(x-0.5)^2} [yshift=3pt] node[pos=.95,left] {\footnotesize{$f_x+(x-0.5)^2$}};    }         
			%           \only<3>       {            \addplot[violet,dashed, very thick] [domain=-1.5:1.7]{abs(0.75-(x-0.5))-1.2*(x-0.5)^2} [yshift=3pt] node[pos=.95,left] {\footnotesize{$f_x-(x-0.5)^2$}};    }        
			%      \draw[fill] (1,1) circle [radius=5];
			
			\addplot [only marks,mark=*] coordinates { (0.5,0.75) };
			%		{   \addplot [only marks, mark=*, mark size=1.5] coordinates { (11/12,0) };}
			\end{axis}		
			\end{tikzpicture}
		\end{center}
		%	\caption{$f(x)=|x^2-1|$}
	\end{subfigure}%
	\begin{subfigure}{.25\textwidth}
		\begin{center}
			\begin{tikzpicture}[scale=0.45]
			\pgfplotsset{every tick label/.append style={font=\large}}
			\begin{axis}[%
			domain = -2:2,
			samples = 200,
			axis x line = center,
			axis y line = center,
			xtick={0.5},
			ytick=\empty
			%        ticks = none
			]		
			\addplot[black, ultra thick] {abs(x^2-1)} [yshift=3pt] ;
			{         \addplot[purple, very thick] {abs(0.75-(x-0.5))} [yshift=6pt] ;      }
			%		       {            \addplot[purple, very thick] [domain=-1:2]{abs(0.75-(x-0.5))+1.2*(x-0.5)^2} [yshift=3pt] node[pos=.95,left] {\footnotesize{$f_x+(x-0.5)^2$}};    }         
			%           \only<3>       {            \addplot[violet,dashed, very thick] [domain=-1.5:1.7]{abs(0.75-(x-0.5))-1.2*(x-0.5)^2} [yshift=3pt] node[pos=.95,left] {\footnotesize{$f_x-(x-0.5)^2$}};    }        
			%      \draw[fill] (1,1) circle [radius=5];
			
			\addplot [only marks,mark=*] coordinates { (0.5,0.75) };
			%		{   \addplot [only marks, mark=*, mark size=1.5] coordinates { (11/12,0) };}
			\end{axis}		
			\end{tikzpicture}
		\end{center}
		%	\caption{$f(x)=|x^2-1|$}
	\end{subfigure}%
	\caption{One-sided models for the function $f(x,\xi)=|x^2-1|$ at the point $x=0.5$, appearing in the same order as in Table~\ref{table:functionmodels}.}
	\label{fig:illustr_lower_model}
\end{figure}
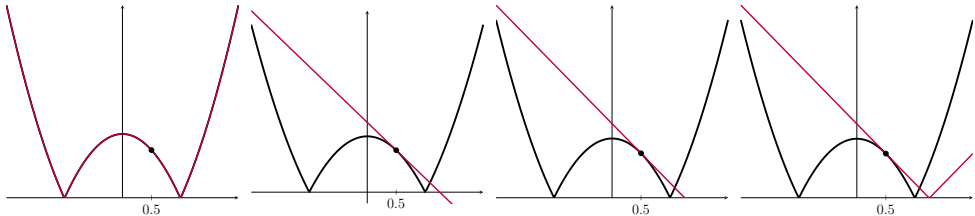

Under the assumptions (a) and (b) above 
and 
standard Lipschitz conditions, 
the paper~\cite{davis2019stochastic}
proves 
that a natural implicit smoothing of the problem---the \emph{Moreau envelope}---  
serves 
as an approximate Lyapunov function 
for 
the algorithm dynamics
and moreover 
its gradient 
tends 
to zero 
at 
a controlled rate.
As argued in~\cite{davis2019stochastic}, 
the size of the Moreau envelope's gradient 
is
a meaningful stationarity measure because it certifies the existence of a nearby point that nearly satisfies first-order necessary conditions for optimality.
Although these results
do not require the objective function
$f$ to be 
smooth 
or 
convex, 
they do require the constraint set $\cX$ to be convex. 
In the current work, 
we 
aim 
to relax the convexity assumptions on $\cX$, while maintaining similar convergence guarantees. 

\subsection{The algorithmic framework}

Moving beyond convexity of $\cX$ immediately yields 
two  challenges. 
The first arises from a conceptual obstruction in extending the proof of~\cite{davis2019stochastic}, 
while the second arises from the practical difficulty of implementing~\eqref{eqn:basic_model_alg} when $\cX$ is nonconvex.
Beginning with the former, 
let us recall 
a central tool for analyzing first-order methods: the three point inequality (see e.g. \cite{beck,nesterov2013gradient,teboulle2018simplified,chen1993convergence}). 
To that end, 
let $g$ 
be 
a $\rho$-weakly convex function 
and 
let $\cX$
be 
a closed convex set. 
Then for any $\bar \rho>\rho$, 
the perturbed function $y\mapsto g(y)+\frac{\bar \rho-\rho}{2}\|y-x\|^2$ 
is 
strongly convex with parameter $\bar \rho-\rho$. 
Therefore its minimizer $\hat x$ over $\cX$ 
satisfies 
the estimate:
\begin{align}\label{eq:3point}
\frac{\bar \rho-\rho}{2}\|y-\hat{x}\|^2\leq \left(g(y) +\frac{\bar \rho}{2}\|y-x\|^2\right)-\left(g(\hat x)+\frac{\bar \rho}{2}\|x-\hat{x}\|^2\right), \qquad \forall y \in \cX
\end{align}
Such three-point estimates 
underpin 
much of complexity analysis in convex optimization, 
and 
they 
are 
similarly crucial to the results of~\cite{davis2019stochastic}, 
where they are applied to the subproblem~\eqref{eqn:basic_model_alg}. The inequality \eqref{eq:3point} also plays an important role in the related works \cite{asi1,asi2} on the stability of stochastic proximal algorithms.

Without convexity of $\cX$, the three point inequality may fail.		
To overcome this difficulty,
we
will restrict 
ourselves to a favorable class 
of 
sets $\cX$, 
namely, 
those that 
are
\emph{$r$-proximally smooth}. Following~\cite{clarkelower}, we say that $\cX$  is $r$-proximally smooth if the nearest-point projection $\proj_{\cX}(\cdot)$ 
evaluates to 
a singleton on the tube $\{x:\dist(x,\cX)<r\}$.\footnote{Proximally smooth sets 
	have appeared 
	under a variety of names in the literature, 
	including \emph{sets with positive reach}~\cite{federer1959curvature} and \emph{uniformly prox-regular sets}~\cite{poliquin1996prox}. 
}  
The class of proximally smooth sets
is 
broad
and 
includes 
all convex sets, 
sublevel sets of weakly convex functions \cite{adly2016preservation},
and compact $C^2$-submanifolds of $\R^d$.
We 
will show that under the proximal smoothness assumption,
an estimate similar to~\eqref{eq:3point} holds 
for 
a sufficiently large value
of 
$\bar \rho$.	
This result should be expected, since weak convexity and proximal smoothness are closely related. 
For example, the epigraph of any weakly convex function is proximally smooth~\cite[Theorem 5.2]{clarkelower}.
In light of this estimate, 
a quick argument  
shows 
that the results of~\cite{davis2019stochastic}
extend 
to proximally smooth sets $\cX$. 
Although this result 
is 
already appealing, 
there 
still remains 
a central practical difficulty: 
when $\cX$ is nonconvex, 
it may be impossible to implement~\eqref{eqn:basic_model_alg}, since it requires solving a  nonconvex problem.

To develop a more easily implementable variant of~\eqref{eqn:basic_model_alg}, 
we 
draw on core techniques 
of  manifold optimization~\cite{absil2009optimization} and nonlinear programming~\cite{NW}. 
Namely, 
we 
replace~\eqref{eqn:basic_model_alg} 
with two simpler steps: 
the first step 
optimizes 
the model function 
over 
a simplified local approximation 
of 
$\cX$, 
while the second ``retracts" this iterate back to $\cX$. 
More formally, 
we
analyze
algorithms
built 
from three basic ingredients. 
The first 
is 
a family 
of 
models $f_x(\cdot,\xi)$
of 
the objective function as outlined previously.
The second 
is 
a family 
of 
local approximations $\cX_x$ 
of 
the constraint set,
indexed by basepoints $x\in \cX$.
The third   
is 
a retraction map $\cR_x\colon\cX_x\to\cX$ 
that 
restores 
feasibility
and
acts
as an approximate nearest point projection onto $\cX$. 
With these ingredients in hand, 
we 	
arrive 
at our main algorithm,
which simply iterates the steps
\begin{equation}\label{eqn:mba}
\begin{aligned}
&\textrm{Sample } \xi_t \sim P\\
& \textrm{Set } \tilde x_{t} = \argmin_{x \in \cX_{x_t}}~ \left\{f_{x_t}(x,\xi_t) + \frac{\beta_t}{2} \|x - x_t\|^2\right\}\\
& \textrm{Set }  x_{t+1} = \cR_{x_t}(\tilde x_t)
\end{aligned}.
\end{equation}
Thus in each iteration, 
the algorithm 
minimizes 
a quadratically regularized stochastic model 
of 
the objective function 
over 
a simplified model 
of 
the constraint set.
Then to restore feasibility,  
it 
``retracts" 
$\tilde x_t$ to $\cX$.

To prove efficiency estimates for~\eqref{eqn:mba}, 
we 
assume 
the building blocks 
$f_x(\cdot,\xi)$, 
$\cX_x$, 
and 
$\cR_x$
behave favorably. 
These assumptions, summarized in Table~\ref{tab:assumpblocks},  come in two flavors: regularity of the individual building blocks and control on their approximation quality. 
Beginning with the models, 
we 
assume properties (a) and (b) as above.
Turning to the constraint approximation,
we first
assume
the sets $\cX_x$
are 
$r$-proximally smooth 
for 
all $x \in \cX$. 
Next, 
we  
assume 
that for any $x \in \cX$, 
the distance function $\dist(y, \cX_x)$ 
minorizes 
$\dist(y, \cX)$ up to quadratic error, near the basepoint.
This condition 
mirrors 
the approximation requirements 
of 
the functional models $f_x(\cdot,\xi)$, 
and 
we 
will see 
it 
holds 
for important  examples. 
Finally 
we 
assume 
the retraction $\cR_x(\cdot)$ 
restores 
feasibility 
in 
a controlled way, 
as suggested 
in Table~\ref{tab:assumpblocks}.

\begin{table}[t!]
	\label{table:intro}
	\begin{center}
		\begin{tabular}{| c | c |c|}
			\hline
			& Regularity & Approximation ($\forall x,y\in\cX,w\in\cX_x$)\\\hline
			Functions $f_x(\cdot,\xi)$  & \parbox[t]{3cm}{$\rho$-weakly convex
				\\
				$L$-Lipshitz}	
			&\parbox[t]{5.7cm}{$\E_{\xi}f_x(x,\xi)=f(x)$\\	
				$\E_{\xi}f_x(y,\xi)\le f(y)+\frac{\osa}{2}\|y-x\|^2$}\\
			\hline
			Sets $\cX_x$ & $r$-proximally smooth 
			&$\dist(y, \cX_{x})  \leq \frac{\accD}{2} \| x- y\|^2$\\
			\hline
			Retraction $\cR_x$& - &$\|w - \cR_x(w) \| \leq \frac{\accR }{2} \|x-w\|^2 $\\
			\hline
		\end{tabular}
	\end{center}
	\caption{Assumptions on the algorithmic building blocks}\label{tab:assumpblocks}
\end{table}

\subsection{Two examples}		
To illustrate the algorithmic setup,
we 
analyze in detail 
two examples 
of 
constraint sets and their set approximations.  
\paragraph{Nonsmooth optimization over smooth manifolds} As the first example, consider a compact $C^{\infty}$ submanifold $\cX$ of $\RR^d$. 
We may choose $\cX_x=\cX$ and $\cR_x = I_d$
or 
we 
may declare $\cX_x$ to be the translated tangent space $x+T_{\cX}(x)$ and $\cR_x$ to be the projection onto $\cM$. See Figure~\ref{fig:sphere} for an illustration. 
In this setup, 
the tangent space approximation 
may yield 
considerably simpler subproblems for computing $\tilde x_t$.
The retraction,
however, 
may still be
somewhat costly. Nonetheless, this set-up already subsumes the stochastic Riemannian subgradient method. Convergence guarantees for the stochastic Riemannian subgradient method were recently obtained in \cite{li2019nonsmooth}. Our paper was developed concurrently and independently of this work. 

\begin{figure}[h!]
	\centering
	\includegraphics[scale=0.5]{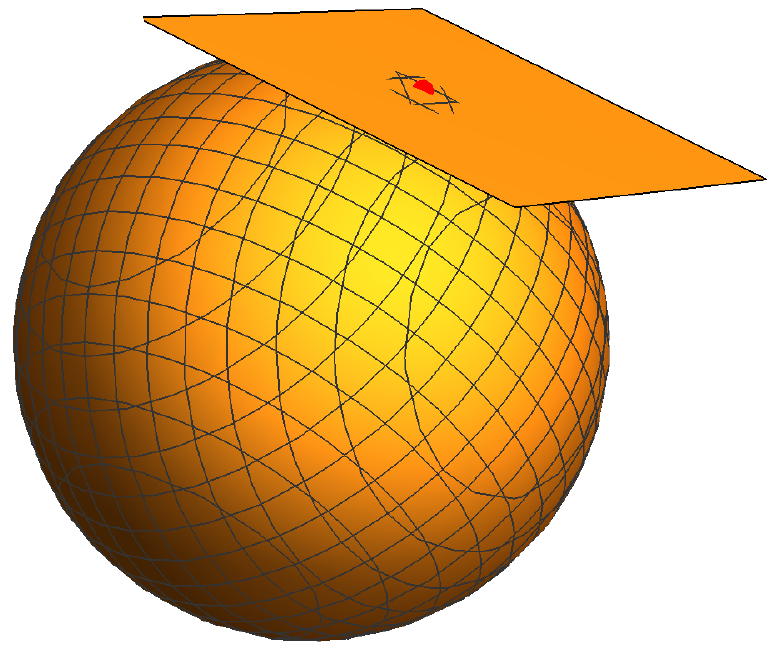}
	\caption{Unit Sphere and its tangent approximations.}\label{fig:sphere}
\end{figure}

\paragraph{Nonsmooth optimization with functional constraints} 
As the second example, 
we 
consider sublevel sets
$\cX=[g\leq 0]$ of closed weakly convex functions $g \colon \RR^d \rightarrow \RR \cup \{+\infty\}$, which satisfy a standard constraint qualification. For instance, sets cut out by smooth nonlinear inequalities $g_i\leq 0$ for $i=1,\ldots,m$ correspond to setting $g=\max_{i=1,\ldots, m} g_i$. 
For such sublevel sets, 
it 
may be
costly to compute a retraction, and therefore
instead we 
seek 
an inner approximation of $\cX$, 
guaranteeing that $\tilde x_t \in \cX$.
A natural class 
of 
such approximations
arises
whenever we 
have 
access to a family of convex two-sided models $g_x(\cdot)$ 
of 
$g$, meaning:
$$|g_x(y)-g(y)|\leq \frac{\gamma}{2}\|y-x\|^2\qquad \forall x,y\in \R^d,$$
where $\gamma>0$ is a constant independent of $x$. See Table~\ref{table:tab} for some notable examples of such two-sided models.
Indeed, then we
may  define
$$\cX_x:=\left\{y: g_x(y)+\frac{\gamma}{2}\|y-x\|^2\leq 0\right\}.$$ 
Since the inclusion $\cX_x \subseteq \cX$ holds for all $x \in \cX$, we may use the identity retraction $\cR_{x} = I_d$. For example, if $\cX$ is cut out by smooth inequalities $$\cX=\{x: g_i(x)\leq 0~~\forall i=0,\ldots,m\},$$ then under reasonable regularity conditions, we may set $$\cX_x=\left\{y:g_i(x)+\langle \nabla g_i(x),y-x\rangle+\frac{\gamma}{2}\|y-x\|^2\leq 0  ~~\forall i=0,\ldots,m\right\},$$
where $\gamma$ is any upper bound on the Lipschitz constants of the gradients $\nabla g_i$. 
 See Figure~\ref{fig:inequality} for an illustration.
\begin{table}[t!]
	\centering
	\begin{tabular}{|c|c|}
		\hline
		Constraint $g(x)$ & Model $g_x(y)$ \\
		\hline
		$g(x)$ & $g(y)+\frac{\gamma}{2}\|y-x\|^2$ \\
		$\displaystyle\max_{i=1,\ldots,m}~g_i(y)$ & $\displaystyle\max_{i=1,\ldots,m}~\{g_i(x)+\langle \nabla g_i(x),y-x\rangle\}$  \\
		$h(c(x))$ & $h(c(x)+\nabla c(x)(y-x))$\\
		\hline
	\end{tabular}
	\caption{Typical  convex two-sided models; the function $g(\cdot)$ is assumed to be $\gamma$-weakly convex, $g_i$ are $C^1$-smooth with Lipschitz gradients, $h(\cdot)$ is Lipschitz and convex, and $c(\cdot)$ is $C^1$-smooth with Lipschitz Jacobian.}\label{table:tab}
\end{table}

\begin{figure}[t]
	\centering
	\begin{subfigure}{.49\textwidth}
		\centering
		\includegraphics[scale=0.35]{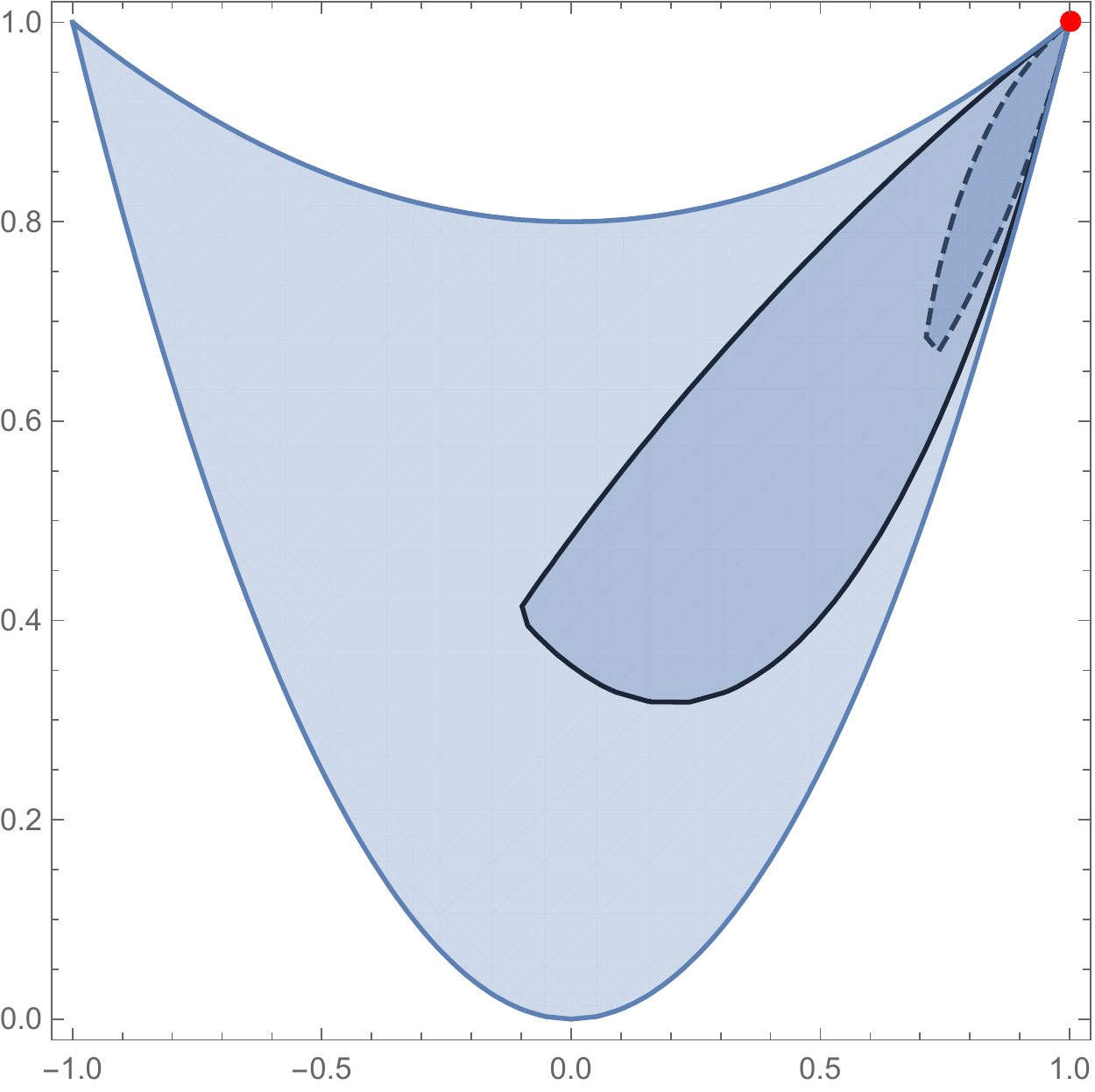}
	\end{subfigure}
	\begin{subfigure}{.49\textwidth}
		\centering
		\includegraphics[scale=0.35]{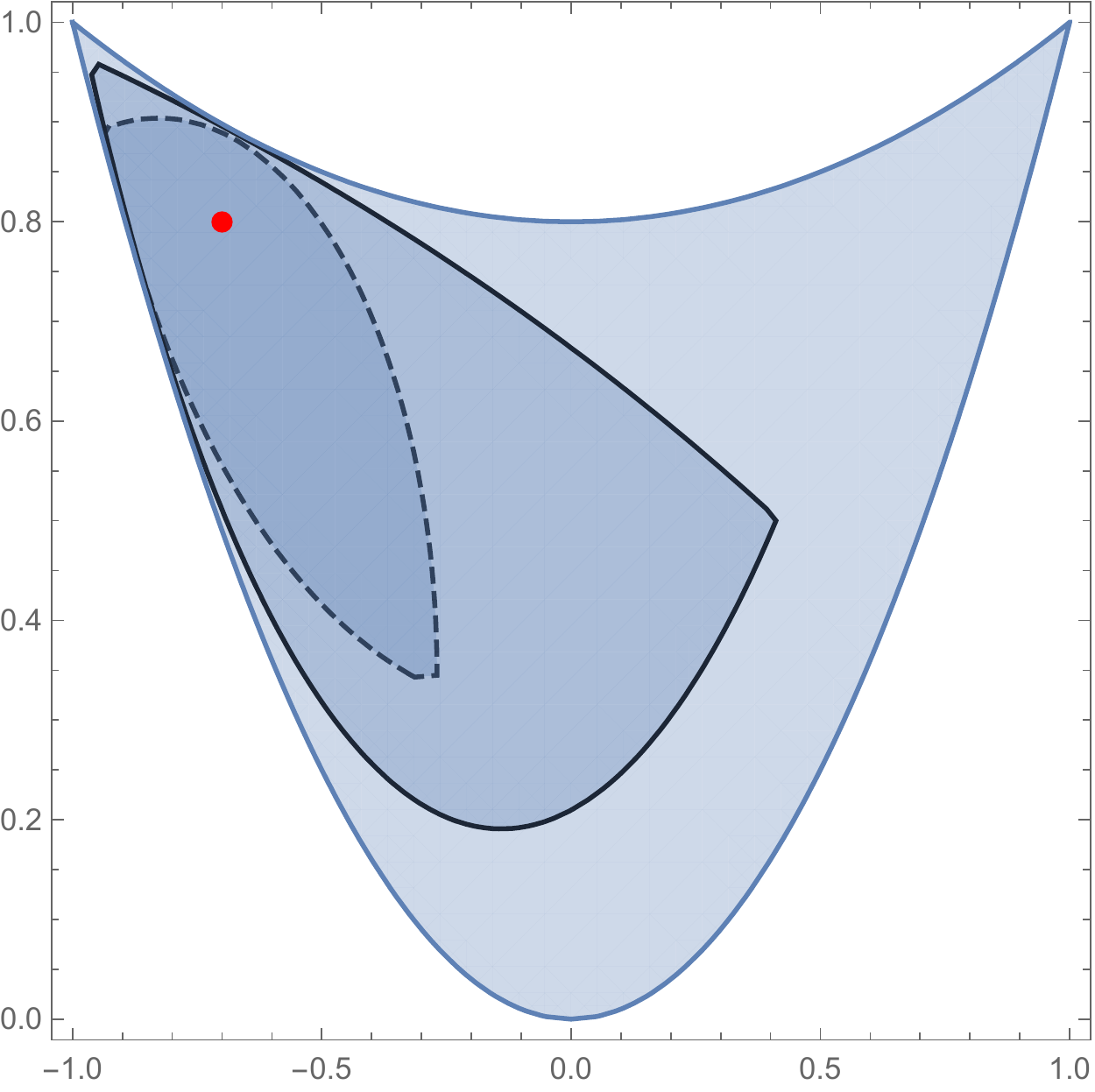}
	\end{subfigure}
	\caption{Both of the figures depict the feasible region $\cX=\{(x,y): g_1(x,y)\leq 0, g_2(x,y)\leq 0\}$, where we define the two quadratics $g_1(x,y)=x^2-y$ and $g_2(x,y)=y - \tfrac{1}{5} x^2-\tfrac{4}{5}$. The basepoints where the set approximations are formed, depicted in red, are $(x_0,y_0)\in \{(1,1),(-0.7,0.8)\}$. The region with the solid  boundary is $\{(x,y):g_i(x,y) + \tfrac{1}{4} \|(x,y)-(x_0,y_0)\|^2\leq 0 ~\forall i=1,2\}$. The region with the dashed boundary is $\{(x,y):g_i(x_0,y_0)+\langle \nabla g(x_0,y_0),(x,y)-(x_0,y_0)\rangle + 1.1\|(x,y)-(x_0,y_0)\|^2\leq 0~\forall i=1,2\}$.}\label{fig:inequality}
\end{figure}

\subsection{Convergence guarantees}
Similarly to the arguments in \cite{davis2019stochastic}, our analysis of Algorithm~\eqref{eqn:mba} will be based on the following three constructions: the {\em Moreau envelope} and the {\em proximal map}, defined respectively by
$$\cM_{\lambda}(x)= \min_{y\in \cX}~ \left\{f(y)+\frac{1}{2\lambda}\|y-x\|^2\right\},\quad\cP_{\lambda}(x):=\argmin_{y\in \cX}~ \left\{f(y)+\frac{1}{2\lambda}\|y-x\|^2\right\},$$
and the {\em stationarity measure}
$$\mathcal{C}_{\lambda}(x):=\lambda^{-1}\cdot\dist(x,\cP_{\lambda}(x)).$$
The quantity  $\mathcal{C}_{\lambda}(x)$ has an intuitive interpretation in terms of near-stationarity for the target problem \eqref{eqn:opt_prob}; see Section~\ref{sec:prob_stat_meas} for details. 
It is worthwhile to mention that $\cC_{\lambda}(x_t)$ can be 
interpreted as the norm of the gradient of the Moreau envelope $\cM_{\lambda}(\cdot)$ when $\lambda$ is sufficiently large, but for the sake of brevity, we will not explore this connection further.
The efficiency guarantees
of
this work
mirror 
those in~\cite{davis2019stochastic}
and 
the core of the argument still remains: 
the Moreau envelope $\cM_{\lambda}$
is 
an approximate Lyapunov function for~\eqref{eqn:mba}
and
the algorithm finds a point $x$ satisfying $\EE\left[\cC_{\lambda}(x)\right] \leq \varepsilon$ after at most $O(\varepsilon^{-4})$ iterations.

\subsection{Related Work} 
There are a number of recent closely related papers. 
\paragraph{Nonsmooth optimization over Riemannian manifolds}
The closely related paper~\cite{li2019nonsmooth} appeared on arXiv a few months prior to ours. The two papers were developed independently and are complementary. The authors of \cite{li2019nonsmooth}
prove a ``Riemannian subgradient inequality" for a weakly convex function over the Stiefel manifold, and remark on its possible extensions to compact manifolds. They use this inequality to analyze the stochastic Riemannian subgradient method, and obtain convergence guarantees that are similar to ours, when specialized to this setting. Other notable lines of work include  gradient sampling \cite{borckmans2014riemannian,grohs2016varepsilon,hosseini2018line,hosseini2017riemannian}, proximal gradient \cite{chen2019proximal}, and the proximal point method \cite{ferreira2002proximal,bacak2016second,de2016new} on Riemannian manifolds.

{\em Proximally Smooth sets.} Proximal smoothness was systematically studied in \cite{clarkelower} with the view towards optimization theory, though the core definition dates back to Federer \cite{federer1959curvature}. To the best of our knowledge, there are only a few papers that study first-order algorithms for minimizing functions over proximally smooth sets. Notable examples include \cite{balashov2019error,balashov2019gradient}, which analyze the efficiency of the projected gradient method on proximally smooth sets.

{\em Optimization with functional constraints.} There have recently appeared a number of papers that develop first-order methods for minimizing smooth (stochastic) objectives over sets cut out by smooth deterministic inequalities \cite{Cartis2014,wang2017penalty,facchinei2017ghost}. The convergence guarantees in these works are all stated in terms of a KKT-residual measure, and therefore are not directly comparable to the ones we obtain here.
The papers~\cite{ma2019proximally} find approximate KKT points when the constraining functions may be nonsmooth, while the papers \cite{boob2019proximal,lin2019inexact} additionally allow stochastic constraints. In particular, the work \cite{boob2019proximal}  explicitely utilizes Lagrange multipliers within the algorithms.

\section{Preliminaries}\label{sec:prelim}
In this section
we 
introduce
the basic notation of the paper, 
which 
mostly follows 
that in 
standard texts 
in 
convex 
and 
variational analysis, such as \cite{RW98,penot_book,Mord_1}.
Throughout the work, 
we 
consider 
the Euclidean space $\mathbb{R}^d$, 
which we 
equip 
with 
a fixed inner product $\langle\cdot,\cdot\rangle$ 
and 
its induced norm $\|\cdot\|=\sqrt{\dotp{\cdot,\cdot}}$.  
We 
let 
the symbol $\BB(x,r)$ 
denote 
the closed ball 
of 
radius $r>0$ 
around 
a point $x$.  
We 
define  
the 
distance 
and 
the nearest-point projection 
of 
a point $y\in\R^d$ 
onto 
a set $\cX\subset\R^d$ 
by
\[
\dist(y,\cX):=\inf_{x\in \cX}~\|x-y\|\qquad \textrm{and}\qquad\proj_{\cX}(y):=\argmin_{x\in \cX}~\|x-y\|,
\]
respectively. 
For any real number $r>0$, 
we
define 
the {\em $r$-tube} 
around 
$\cX$ 
to be the set
$
U_{\cX}(r):=\{y\in\R^d:\dist(y,\cX)<r\}.
$
For 
any function $g\colon\R^d\to\R\cup\{\infty\}$ 
and 
$r\in\R$, 
we 
define 
the sublevel set $[g\leq r]:=\{x: g(x)\leq r\}$. For ant set $Q\subset\R^d$, the indicator function, denoted by $\delta_Q(\cdot)$ is defined to be zero on $Q$ and $+\infty$ off it.

Consider a function $f\colon\R^d\to\R\cup\{\infty\}$ and a point $\bar x$, with $f(\bar x)$ finite. 
The {\em subdifferential} of $f$ at $\bar x$, denoted by $\partial f(\bar x)$, consists of all vectors $v\in\R^d$ satisfying
$$f(x)\geq f(\bar x)+\langle v,x-\bar x\rangle+o(\|x-\bar x\|)\qquad \textrm{as } x\to \bar x.$$ 
We will often use the following basic variational construction. The \textit{normal cone} to a set $\cX$ at a point $\bar x\in \cX$, denoted $N_{\cX}(\bar x)$, consists of all vectors $v\in\R^d$ satisfying 
\begin{equation}\label{eqn:frech_norm}
\langle v,x-\bar x\rangle\leq o(\|x-\bar x\|) \qquad \textrm{as }x\to \bar x \textrm{ in }\cX.
\end{equation}

\section{Weak convexity, proximal smoothness, and three-point inequality}

In this section, 
we
introduce
and
relate
the notions of weak convexity and proximal smoothness. 
We
then connect 
these properties to our key algorithmic tool: the three-point inequality. 
%Both of these notions are simple to define. 
First, we say that a function $f\colon\R^d\to\R\cup\{\infty\}$ is {\em $\rho$-weakly convex}, for some $\rho>0$, if the perturbed function $x\mapsto f(x)+\frac{\rho}{2}\|x\|^2$ is convex. Second, following \cite{clarkelower}, we say that a closed set $\cX$ is {\em $R$-proximally smooth} if the projection $\proj_{\cX}(x)$ is a singleton whenever $\dist(x,\cX)<R$. Although these conditions appear at first unrelated, they are deeply connected in the sense that weak convexity of a function $f$ is essentially equivalent to proximal smoothness of its epigraph. For a precise statement, we refer the reader to 
\cite[Section 5]{clarkelower}.

A unifying theme
of
both 
weakly convex functions
and
proximally smooth sets
is 
uniformity in subdifferential and normal cone constructions. 
These estimates
in 
turn 
lead to
the existence of a three point inequality, as we will soon see. 
For example, 
if $f$ 
is 
a $\rho$-weakly convex function, 
then the inequality
\begin{equation}\label{eqn:subgrad_ineq}
f(y)\ge f(x)+\dotp{v,y-x}-\frac{\rho}{2}\|y-x\|^2,\qquad \textrm{holds for all }x,y \in \R^d,~v\in \partial f(x).
\end{equation}
Likewise, 
the normal cone
of
a proximally smooth set
enjoys
a similarly uniform estimate, as the following result shows. The details can be found in \cite{clarkelower}.
\begin{lemma}[Proximally smooth sets]\label{lem_basic_prop_man}
	Consider an $R$-proximally smooth set $\cX$. The following are true.
	\begin{enumerate}
		\item {\bf (Lipschitz projector)} For any real $r\in (0, R)$ the estimate holds:
		$$\|\proj_{\cX}(x)-\proj_{\cX}(y)\|\le \frac{R}{R-r}\|x-y\|\qquad \forall x,y\in U_{\cX}(r).$$
		%		\item {\bf (Projections \& normals)} For any points $x\in\cX$, $y\in U_{\cX}(R)$, the equivalence holds:
		%		$$x=\proj_{\cX}(y)\qquad\Longleftrightarrow\qquad y-x\in N_{\cX}(x).$$ 
		\item\label{it3:lem} {\bf (Uniform normal inequality)} For any  point $x\in \cX$ and a normal $v\in N_{\cX}(x)$,  the inequality
		\begin{equation}\label{ineq:proximal_normal}
		\left\langle v,y-x\right\rangle\le \frac{\|v\|}{2R}\cdot \|y-x\|^2 \qquad\textrm{holds for all } y\in \cX.
		\end{equation}
	\end{enumerate}
\end{lemma}	

With the uniformity properties in hand, we turn to the \emph{three point inequality}: a key tool in the  complexity analysis of algorithms for convex optimization. Namely, consider a convex function $f$, a convex set $\cX$, and an arbitrary point $x$. Then the  point $\hat x=\argmin_{y\in\cX} f(y)+\frac{\beta}{2}\|y-x\|^2$ satisfies the estimate:
$$\frac{\beta}{2}\|y-\hat{x}\|^2\leq \left(f(y)+\frac{\beta}{2}\|y-x\|^2\right)-\left(f(\hat x)-\frac{\beta}{2}\|x-\hat{x}\|^2\right)\qquad \forall y\in\cX.$$ 
Indeed, this estimate is a direct consequence of strong convexity of $f+\frac{\beta}{2}\|\cdot-x\|^2$ and convexity of $\cX$.  
More generally, one can replace the square Euclidean norm with a Bregman divergence and maintain a similar inequality; see e.g. \cite{beck,nesterov2013gradient,teboulle2018simplified,chen1993convergence}. 
We will see in Lemma~\ref{lem:3_pt_ineq} that when $f$ is weakly convex and $\cX$ is proximally smooth, a three point inequality is still valid---an essential ingredient of the forthcoming results. 
We first record the following lemma, which we will routinely use in the sequel.
We omit the proof since it follows directly from definitions.
\begin{lemma}\label{lem:prox_point_bound}
	Consider a function $g\colon \RR^d \to \RR$ that is $L$-Lipschitz continuous on a set $\cX\subseteq \RR^d$. Then for any point $x \in \cX$,  every point $\displaystyle \hat x \in \argmin_{y \in \cX} \left\{g(y) + \frac{\beta}{2} \|x - y\|^2\right\}$  satisfies the estimate $
	\|\beta(x - \hat x)\| \leq 2L.
	$
\end{lemma}

\begin{lemma}[Three-point inequality]\label{lem:3_pt_ineq}
	Let $f\colon\R^d\to\R$ be a $\rho$-weakly convex function and let $\cX\subseteq\R^d$ be an $R$-proximally smooth set. Suppose  that there exists $L>0$ such that $f$ is $L$-Lipschitz continuous on some neighborhood $U$ of $\cX$. 
	Fix a real $\beta>\rho+\frac{3L}{R}$, a point $x\in U$, and define the  point 
	$\tilde{x}:=\argmin_{y\in \cX}~f(y)+\frac{\beta}{2}\|y-x\|^2.$ Then for any point $y\in U$, the following inequality holds:
	\[
	f(y)-f(\tilde{x})\geq \frac{\beta-\rho-\frac{3L}{R}}{2}\|y-\tilde{x}\|^2+\frac{\beta}{2}\|x-\tilde{x}\|^2-\frac{\beta}{2}\|y-x\|^2.
	\]
\end{lemma}
\begin{proof}
	Appealing to Lemma~\ref{lem_basic_prop_man}(\ref{it3:lem}), we deduce
	$
	\left\langle w,y-z\right\rangle\leq \frac{\|w\|}{2R}\cdot\|y-z\|^2,
	$
	for all $z,y\in\cX$ and $w\in N_{\cX}(z)$.
	Adding this estimate to the subgradient inequality \eqref{eqn:subgrad_ineq} for $f$ yields
	\begin{equation}\label{eqn:proximal_ineq}
	f(y)\geq f(z)+\langle v+w,y-z\rangle-\frac{1}{2}\left(\rho+\frac{\|w\|}{R}\right)\|y-z\|^2\qquad \forall z,y\in\cX, ~v\in \partial f(z).
	\end{equation}
	Now set $z:=\tilde{x}$ and note the equality 
	$\beta(x-\tilde{x})=v+w$ for some vectors $v\in\partial f(\tilde{x})$ and $w\in N_{\cX}(\tilde{x})$, which follows by first order optimality conditions. 
	%Define now $\zeta:=\proj_{T_{\cX}(\hat x)}(v')$ and $w:=\proj_{N_{\cX}(\hat x)}(v')+w'$. Observe the expressions
	%$$\zeta+w=v'+w'=\bar\rho(x-\hat{x})\qquad \textrm{and}\qquad w\in N_{\cX}(\hat x).$$
	Completing the square in the right-hand-side of \eqref{eqn:proximal_ineq} then yields 
	\begin{equation}\label{eqn:3point_basic}
	\begin{aligned}
	f(y)
	&\geq f(\tilde{x})-\frac{\beta}{2}\|y-x\|^2+\frac{\beta}{2}\|x-\tilde{x}\|^2+\frac{(\beta-\rho)-\|w\|/R}{2}\|y-\tilde{x}\|^2.
	\end{aligned}
	\end{equation}
	To complete the proof, we upper bound $\|w\|$ using the triangle inequality, 
	$\|w\|\leq \|v\|+\beta\|x-\tilde{x}\|\leq L+\beta\|x-\tilde{x}\|\leq 3L,$ where the second inequality follows from Lipschitz continuity of $f$ and last inequality follows from Lemma~\ref{lem:prox_point_bound}. 
\end{proof}

Before passing to our algorithmic framework,
we
briefly comment
on the prevalence of weakly convex functions and proximally smooth sets. 
First,
many sets 
of 
practical interest
are
proximally smooth, 
including closed convex sets, sublevel sets of weakly convex functions (assuming a constraint qualification\cite{adly2016preservation}), and compact $C^2$-submanifolds of $\RR^d$.
Weakly convex functions, in turn, are widespread in applications and are typically easy to recognize. One common source to keep in mind is the composite function class, $f(x):=h(c(x)),$
where $h\colon \R^m \to \R$ is convex and $L$-Lipschitz and $c\colon\R^d\to \R^m$ is a $C^1$-smooth map with $\beta$-Lipschitz continuous Jacobian. A quick argument shows that the composite
function $f$ is $L\beta$-weakly convex; for a short argument see e.g. \cite[Lemma 4.2]{comp_DP}.  In particular, 
a variety 
of 
practical problems in statistical signal recovery
are 
weakly convex; see~\cite{davis2019stochastic,charisopoulos2019composite,charisopoulos2019low,li2018nonconvex} for examples.

\section{Main results}\label{sec:prob_stat_meas}
Henceforth, we consider the optimization problem
\begin{equation}\label{eqn:target_prob}
\min ~ f(x)\quad \textrm{subject to}\quad x\in \cX,
\end{equation}
for some closed function $f\colon\R^d\to\R$ that is bounded from below on a closed set $\cX\subseteq\R^d$. We will place further assumptions on $f$ and $\cX$ shortly.

Similar to the arguments in \cite{davis2019stochastic}, 
we 
base
our analysis 
of 
algorithms 
for 
the problem \eqref{eqn:target_prob} 
on
the following three constructions. 
Namely, 
define 
the {\em Moreau envelope} 
and 
the {\em proximal map}, respectively, by
$$\cM_{\lambda}(x)= \min_{y\in \cX}~ \left\{f(y)+\frac{1}{2\lambda}\|y-x\|^2\right\},\qquad\cP_{\lambda}(x):=\argmin_{y\in \cX}~ \left\{f(y)+\frac{1}{2\lambda}\|y-x\|^2\right\},$$
and define the {\em stationarity measure}:
$$\mathcal{C}_{\lambda}(x):=\lambda^{-1}\cdot\dist(x,\cP_{\lambda}(x)).$$
As 
we alluded to 
in the introduction, 
the quantity  $\mathcal{C}_{\lambda}(x)$ 
has 
an intuitive interpretation 
in 
terms 
of 
near-stationarity 
for 
the target problem \eqref{eqn:target_prob}. 
Namely,
fix a point $x\in \R^d$. 
Then the very definition 
of
the  measure $\mathcal{C}_{\lambda}(x)$ 
guarantees 
\begin{equation*}
\left\{\begin{aligned}
\|\hat x-x\|&=\lambda\cdot \cC_{\lambda}(x)\\
f(\hat x)&\leq f(x)\\
\dist(0, \partial (f+\delta_{\cX})(\hat{x}))&\leq \cC_{\lambda}(x)
\end{aligned}\right\},
\end{equation*}
where $\hat x\in \cP_{\lambda}(x)$ is the closest point to $x$.
Thus a small value $\mathcal{C}_{\lambda}(x)$ 
implies 
that $x$ 
is 
{\em near} some point $\hat x$ 
that 
is 
{\em nearly stationary} 
for 
the problem.
The forthcoming analysis 
of 
Algorithms~\ref{alg:stoc_prox_noret} 
and 
\ref{alg:stoc_prox} 
will establish 
a rate at which  $\cC_{\lambda}(x_t)$ 
tends to 
zero 
along 
the iterate sequence $\{x_t\}$. 
It 
is 
worthwhile 
to mention 
that if $f$ is weakly convex and $\cX$ is proximally smooth, then $\cC_{\lambda}(x_t)$ can be 
interpreted as the norm of the gradient of the Moreau envelope $\cM_{\lambda}(\cdot)$ when $\lambda$ is sufficiently large. For the sake of brevity, however, we will not explore this connection further.

\subsection{Model-based minimization over proximally smooth sets}\label{sec:simpl_meth}
In this section, 
we 
present 
an algorithm 
for 
the problem~\eqref{eqn:target_prob},
which 
directly extends
the model-based algorithm 
in \cite{davis2019stochastic} 
to 
the setting 
when 
the constraint set $\cX$ 
is 
not convex, 
but is instead proximally smooth.
The procedure, 
summarized as Algorithm~\ref{alg:stoc_prox_noret}, 
assumes 
access 
to 
a family 
of 
stochastic models $f_{x}(\cdot,\xi)$ 
of 
the objective function 
indexed 
by the base point $x\in \cX$ 
and 
random elements $\xi\sim P$. 
In each iteration $t$, 
the algorithm 
samples
a stochastic model $f_{x_t}(\cdot,\xi_t)$ 
of
the objective,
centered  
at the current iterate $x_t$.
Then it declares the next iterate $x_{t+1}$ 
to be 
the minimizer 
of 
the function $f_{x_t}(\cdot,\xi_t)+\frac{\beta_t}{2}\|\cdot-x_t\|^2$ 
over 
$\cX$.

\begin{algorithm}[H]
	{\bf Input:} initialization $x_0\in \R^d$,   a sequence $\beta_t>0$,  and an iteration count $T\in\mathbb{N}$.
	\break{\bf Step } $t=0,\ldots,T$:	
	%	\begin{equation*}
	\begin{align}
	&\textrm{Sample } \xi_t \sim P\notag\\
	& \textrm{Choose }  x_{t+1} \in \argmin_{x \in \cX}~ \left\{f_{x_t}(x,\xi_t) + \frac{\beta_t}{2} \|x - x_t\|^2\right\}\label{eqn:subprob_to_solve}%\\
	%	& \textrm{Set }  x_{t+1} = \cR_{x_t}(\tilde x_t)
	\end{align}
	\caption{Stochastic Model Based Algorithm
		%: PSSM($x_0,T,\{\alpha_t\}$)
	}
	\label{alg:stoc_prox_noret}
\end{algorithm}
\smallskip

The success 
of 
Algorithm~\ref{alg:stoc_prox_noret} 
relies 
not only  
on the control 
of 
the error $\EE[f_{x}(y,\xi)-f(y)]$, 
but  		  
on the regularity 
of
the models $f_{x_t}$ (weak-convexity) 
and 
of 
the constraint set $\cX$ (proximal-smoothness). Henceforth, we impose the following assumptions.

\begin{assumption}%\label{asmpt:A}
	{\rm
		Fix a probability space $(\Omega,\mathcal{F},P)$ and equip $\R^d$ with the Borel $\sigma$-algebra. 
		We assume that there exist real $\osa, \eta, L, R>0$ satisfying the following properties. 
		\begin{enumerate}
			\item[(A1)] {\bf (Sampling)} It is possible to generate i.i.d.\ realizations $\xi_1,\xi_2, \ldots \sim P$.
			\item[(A2)] {\bf (One-sided accuracy)} There is an open set $U$ containing $\cX $ and a measurable function $(x,y,\xi)\mapsto f_x(y,\xi)$, defined on $U\times U\times\Omega$, satisfying  $$\EE_{\xi}\left[f_x(x,\xi)\right]=f(x)
			\quad \textrm{and} 
			\quad\EE_{\xi}\left[f_x(y,\xi)-f(y)\right]
			\leq\frac{\osa}{2}\|y-x\|^2\qquad \forall x,y\in U.$$
			\item[(A3)] {\bf (Proximal smoothness and weak-convexity)} The set $\cX$ is $R$-proximally smooth and the function $f_x(\cdot,\xi)$ is $\eta$-weakly convex $\forall x\in U$, a.e. $\xi\sim P$.

			\item[(A4)] {\bf (Lipschitz property)} $f$ is $L$-Lipschitz continuous on $\cX$; and for all $x\in \cX$ and a.e.\ $\xi\sim P$, the function $f_x(\cdot, \xi)$ is $L$-Lipschitz continuous on some neighborhood of $\cX$.
	\end{enumerate}}
\end{assumption}

The assumptions $(A1)-(A4)$ 
are 
almost identical 
to 
the ones 
used 
in \cite{davis2019stochastic},
except that $\cX$ 
is not required to be 
convex, 
but only proximally smooth. 
In particular, 
there 
is 
a wide variety 
of 
models $f_x(\cdot,\xi)$ one 
can use
within Algorithm~\ref{alg:stoc_prox_noret}. 
Table~\ref{table:functionmodels} 
and 
Figure~\ref{fig:illustr_lower_model} 
outline 
a few possibilities. 
Properties (A1)-(A3) are verified for  the proximal point, subgradient, and prox-linear models  in \cite[Section 4.3]{davis2019stochastic}, and they are verified for the clipped subgradient model  in \cite{asi2}. 
The class of proximally smooth sets is 
broad 
and 
includes 
all convex sets, 
sublevel sets of weakly convex functions \cite{adly2016preservation},
and 
compact $C^2$-submanifolds of $\R^d$.

The following theorem summarizes convergence guarantees for Algorithm~\ref{alg:stoc_prox_noret} that directly parallel and generalize the results in \cite{davis2019stochastic}. 

\begin{theorem}[Convergence guarantees]\label{thm:simple_result}
	Define the constant $\gamma:=\eta+\frac{3L}{R}$ and fix a real $\bar{\rho} > \gamma + \osa$ and a stepsize sequence $\beta_t \in (\bar\rho, \infty)$. Let $\{x_t\}_{t=0}^T$ be the iterates generated by  Algorithm~\ref{alg:stoc_prox_noret} and let  $t^*\in \{0,\ldots,T\}$ be sampled according to the discrete probability distribution
	$\mathbb{P}(t^*=t)\propto\frac{1}{\beta_t - \eta-\frac{3L}{R}}.$
	Then the point $x_{t^*}$ satisfies the estimate:	
	%	Let $\{x_t\}_{t=0}^T$ be the iterates produced by Algorithm~\ref{alg:stoc_prox}. Then the estimate holds:
	$$\EE[\cC_{1/\bar{\rho}}(x_{t^*})^2]\leq \frac{\bar{\rho}(f(x_{0})-\min_{\cX} f)+  \frac{\bar \rho^2L^2}{2}\cdot\sum_{t=0}^T\frac{1}{\beta_t(\beta_t - \gamma)}}{\sum_{t=0}^T \frac{\bar \rho - \gamma - \osa}{\beta_t-\gamma}}.$$
	%	where $t^*\in [0,T]$ is a random index chosen according to ${\rm Pr}(t^*=t)\propto\frac{1}{\beta_t-\gamma}$.	
	%	Then the point $x_t^*$ returned by Algorithm~\ref{alg:lower_model} satisfies:
	In particular, if Algorithm~\ref{alg:stoc_prox} sets $\bar \rho=2(\eta+\osa+\frac{3L}{R})$ and $\beta_t=\max\left\{\gamma,\sqrt{\frac{\bar{\rho}L^2(T+1)}{2\Delta}}\right\}$, for some real $\Delta \ge f(x_0) - \min_{\cX} f$, then the estimate holds:
	$$\EE[\cC_{1/\bar{\rho}}(x_{t^*})^2]\leq \max\left\{\frac{\bar \rho \Delta }{T+1}, L\sqrt{\frac{2\bar \rho \Delta }{T+1}}\right\}.$$
\end{theorem}

The following section develops a proof of Theorem~\ref{thm:simple_result} by combining the arguments in \cite{davis2019stochastic} with the three-point inequality proved in Lemma~\ref{lem:3_pt_ineq}.

\subsubsection{Proof of Theorem~\ref{thm:simple_result}}
In what follows, we let   $x_0,\ldots, x_t$ and $\xi_0,\ldots, \xi_{t}$ be the iterates and random elements generated by Algorithm~\ref{alg:stoc_prox}. We will use the shorthand $\EE_t$ to denote the expectation conditioned on $\xi_0,\ldots, \xi_{t-1}$.  We begin with a key one-step improvement lemma. 

\begin{lemma}[One-step improvement]\label{lem:one_step}
	Fix a real $\bar{\rho}>0$, define the constant $\gamma:=\eta+\frac{3L}{R}$ and a stepsize sequence $\beta_t >\gamma$. Then for every iterate $t$ and any proximal  point $\hat x_t\in\cP_{1/\bar \rho}(x_t)$, the estimate holds:  
	\begin{equation}\label{eqn:one_step_improv}
	\EE_t\left[\|\hat x_t-x_{t+1}\|^2\right]
	\leq \frac{\beta_t+\osa-\bar \rho}{\beta_t-\gamma}\|\hat x_t-x_t\|^2+ \frac{L^2}{\beta_t(\beta_t-\gamma)}.
	\end{equation}
\end{lemma}
\begin{proof}
	Fix an index $t$, and choose a proximal point
	$\hat x_t\in\cP_{1/\bar \rho}(x_t)$. Appealing to Lemma~\ref{lem:3_pt_ineq} with $f_{x_t}(\cdot,\xi_t)$ in place of $f$, and with $x=x_t$ and $y=\hat x_t$, we deduce
	\begin{align}
	\EE_t&\left[\frac{\beta_t-\gamma }{2}\|{\hat x}_t-x_{t+1}\|^2+\frac{\beta_t}{2}\|x_t-x_{t+1}\|^2-\frac{\beta_t}{2}\|{\hat x}_t-x_t\|^2\right]\notag
	\\
	&\leq \EE_t\left[f_{x_t}({\hat x}_t,\xi_t)-f_{x_t}(x_{t+1},\xi_t)\right]\notag\\
	&\leq \EE_t\left[f_{x_t}({\hat x}_t,\xi_t)-f_{x_t}(x_{t},\xi_t) \right]+L\EE_t[\|x_t-x_{t+1}\|]\label{eqn:it1_nee}\\
	&\leq f({\hat x}_t)-f(x_t)+\frac{\osa}{2}\|{\hat x}_t-x_t\|^2+L\EE_t\left[\|x_t-x_{t+1}\|\right]\label{eqn:it2_nee}\\
	&\leq  \frac{\osa-\bar \rho}{2}\|{\hat x}_t-x_t\|^2+L\EE_t\left[\|x_t-x_{t+1}\|\right]\label{eqn:it3_nee},
	%&\leq \frac{\tau-\bar \rho}{2}\|x-x_t\|^2+L\EE\left[\|x_t-x_{t+1}\|\right]+\frac{\rho}{2}\EE[\|x_t-x_{t+1}\|^2]
	\end{align}
	where \eqref{eqn:it1_nee} uses (A4), the estimate \eqref{eqn:it2_nee} uses (A2), and \eqref{eqn:it3_nee} follows from the definition of $\hat{x}_t$ as the proximal point.
	Setting $\delta=\sqrt{\EE\|x_t-x_{t+1}\|^2}$ and rearranging yields  
	%\begin{align*}
	%&\EE_t\left[\frac{\beta_t-\gamma_t }{2}\|x-x_{t+1}\|^2+\frac{\beta_t}{2}\|x_t-x_{t+1}\|^2-\frac{\beta_t}{2}\|x-x_t\|^2\right]\\
	%&\leq \frac{\tau-\bar \rho}{2}\|x-x_t\|^2+L\EE\left[\|x_t-x_{t+1}\|\right]+\frac{\rho}{2}\EE[\|x_t-x_{t+1}\|^2]
	%\end{align*}
	\begin{align*}
	\EE_t\left[\frac{\beta_t-\gamma }{2}\|\hat x_t-x_{t+1}\|^2\right]&\leq \frac{\beta_t+\osa-\bar \rho}{2}\|\hat x_t-x_t\|^2+L\delta-\frac{\beta_t}{2}\delta^2\\
	&\leq \frac{\beta_t+\osa-\bar \rho}{2}\|\hat x_t-x_t\|^2+ \frac{L^2}{2\beta_t},
	\end{align*}
	where the last inequality follows from maximizing the expression $L\delta-\frac{\beta_t}{2}\delta^2$ in $\delta\in \R$.
	This completes the proof of the lemma.
\end{proof}

The convergence guarantees now quickly follow.

\begin{proof}[Proof of Theorem~\ref{thm:simple_result}]
	Fix an iteration $t$ and a point  $\hat x_t\in\cP_{1/\bar \rho}(x_t)$. Then
	using the definition of the Moreau envelope and appealing to  Lemma~\ref{lem:one_step}, we deduce  
	\begin{align*}
	\EE_t[ \cM_{1/\bar \rho}(x_{t+1})] &\leq  \EE_t\left[f( \hat x_{t})  +\frac{\bar \rho}{2}\cdot \| x_{t+1} - \hat x_t \|^2\right]\\
	&\leq f( \hat x_{t}) +\frac{\bar \rho}{2}\left[\|\hat x_t-x_{t}\|^2+\left(\tfrac{\beta_t+\osa-\bar \rho}{\beta_t-\gamma}-1\right)\|\hat x_t-x_{t}\|^2  + \frac{L^2}{\beta_t(\beta_t-\gamma)}\right]\\
	&= \cM_{1/\bar \rho}(x_{t}) - \frac{\bar \rho- \gamma-\osa}{\bar \rho(\beta_t-\gamma)}\cC_{1/\bar \rho}(x_t)^2  + \frac{\bar\rho L^2}{2\beta_t(\beta_t-\gamma)}.
	\end{align*}
	Taking expectations, iterating the inequality, and using the tower rule  yields:
	%	Using the inequality $\varphi_{1/\bar \rho}(x_{t+1})\geq \min \varphi$ and rearranging, we arrive at
	$$\sum_{t=0}^T \frac{\bar \rho - \gamma - \osa}{\beta_t-\gamma}\EE[\cC_{1/\bar{\rho}}(x_t)^2]\leq \bar{\rho}(\cM_{1/\bar \rho}(x_{0})-\min \cM_{1/\bar\rho})+  \frac{\bar \rho^2 L^2}{2}\cdot\sum_{t=0}^T\frac{1}{\beta_t(\beta_t - \gamma)}.$$
	Dividing through by $\sum_{t=0}^T \tfrac{\bar \rho - \gamma - \osa}{\beta_t-\gamma}$ and recognizing the left side as
	$\EE[\cC_{1/\bar{\rho}}(x_{t^*})^2]$ completes the proof.
\end{proof}

\section{Set-approximations}
One deficiency 
	of 
Algorithm~\ref{alg:stoc_prox_noret} 
	is 
that even if the models $f_x(\cdot,\xi)$ 
are 
convex, 
the subproblems \eqref{eqn:subprob_to_solve} 
could be
nonconvex, 
since $\cX$ is a nonconvex set. 
In Section~\ref{sec:main_results}, 
	we will generalize 
Algorithm~\ref{alg:stoc_prox_noret} 
	by 
allowing one 
to replace 
$\cX$ 
	in 
the subproblem \eqref{eqn:subprob_to_solve} by a close approximation. 
To this end, we introduce the following definition.

\begin{definition}[Set-approximation]
	{\rm Consider a set $\cX\subset\R^d$ and a collection of sets $\cX_x\subset\R^d$ and functions $\cR_x\colon\cX_x\to\cX$, indexed by points $x\in \R^d$. We say that the collection $\{(\cX_x,\cR_{x})\}_{x\in\cX}$ is a {\em set approximation of $\cX$ with parameters $(R,\accD,r_1,\accR ,r_2)$}  if the following conditions hold:
		\begin{enumerate}
			\item [$(i)$] {\bf (Accuracy)} For every $x \in \cX$, the set $\cX_x$ is $R$-proximally smooth and satisfies
			\begin{equation}\label{eqn:need_this_eqn}
			\dist(y, \cX_{x})  \leq \frac{\accD}{2} \| x- y\|^2 \qquad \forall y \in \cX\cap \BB(x,r_1).
			\end{equation}
			%			\textcolor{red}{Why need this???}
			
			\item[$(ii)$] {\bf (Retraction)} For every $x \in \cX$, the mapping $\cR_x$ satisfies
			$$
			\|y - \cR_x(y) \| \leq \frac{\accR }{2} \|x-y\|^2 \qquad \forall y \in \cX_x \cap \BB(x,r_2).
			$$
		\end{enumerate}
	}
\end{definition}

Thus, assumption $(i)$ 
asserts 
that 
$\cX_x$
is 
an outer approximation 
of 
$\cX$ up to quadratic error, while
assumption $(ii)$
asserts 
that $\cR_x$ 
restores 
feasibility, 
while deviating from the identity map 
by 
at most a quadratic error. 
The conditions 
	are 
symmetric 
	when $\cR_x$ 
		is 
	the projection onto $\cX$.

We next discuss two illustrative examples of set approximations:
tangent space approximations of Riemannian manifolds 
	and 
inner approximations to sublevel sets of nonsmooth functions.

\subsection{Riemannian manifolds}
{\rm
State and prove a theorem.
	As the first example,
	suppose that 
	$\cX\subset\R^d$ 
	is 
	a compact $C^{\infty}$-smooth manifold, 
	with Riemannian metric induced by the Euclidean inner-product. 
	There 
	are
	two natural set-approximations for such sets. 
	First,
	it 
	is 
	well-known that $\cX$ 
	is 
	itself $R$-proximally smooth 
	for 
	some $R>0$. 
	Therefore we 
	may simply use 
	$\cX_x=\cX$ as the set approximation.
	In this case, 
	we 
	may set 
	$\cR_x={\rm Id}$, 
	$r_1=r_2=\infty$, 
	and 
	$\accD=\accR =0$ 
	in assumptions $(i)$ and $(ii)$. 
	Notice that the subproblems 
	solved by 
	Algorithm~\ref{alg:stoc_prox} 
	are 
	typically not convex for any $\beta_t$. 
	Consequently, 
	it may be 
	more convenient 
	to choose 
	the tangent space approximations $\cX_x=x+T_{\cX}(x)$. Then any retraction 
	in 
	the sense 
	of 
	manifold optimization \cite{absil2009optimization} automatically 
	satisfies  
	$(ii)$ 
	for 
	some 
	$\accR $ 
	and 
	$r_2$.  	For example, 
	one 
	may set 
	$\cR_x$ 
	to be 
	the nearest point projection 
	from 
	$x+T_{\cX}(x)$ 
	onto 
	$\cX$. Moreover it is straightforward to see that $(i)$ also holds 
	automatically 
	for 
	some 
	$\accD$ 
	and 
	$r_1$. We provide a quick proof sketch for completeness.
	
\begin{lemma}
If 	$\cX\subset\R^d$ is a compact $C^{\infty}$-smooth manifold, then property $(i)$ holds for the set approximations $\cX_x=x+T_{\cX}(x)$ for some $r_1,\tau_1>0$.
\end{lemma}
\begin{proof}
Fix a point $\bar x\in \cX$. By compactness, it suffices to verify \eqref{eqn:need_this_eqn} just around $\bar x$ for some $r_1,\tau_1>0$, which may depend on $\bar x$. To this end, there exists a neighborhood $U$ around $\bar x$ such that the projection map $\proj_{\cX}(\cdot)$ is $C^{\infty}$-smooth on $U$, and its derivative at $\bar{x}$ is the orthogonal projection onto the  $T_{\cX}(\bar x)$. The inverse function theorem therefore guarantees that the restriction $\mathcal{P}\colon \cX_{\bar x}\to \cX$ of $\proj_{\cX}(\cdot)$ is a local $C^{\infty}$-diffeomorphism around the origin. Consequently for all $z\in \cX_{\bar x}$ near $\bar x$, we have
\begin{align*}
\mathcal P(z)&=\mathcal P(\bar x)+\proj_{T_{\cX}(\bar x)}(z-\bar x)+O(\|z-\bar x\|^2)=z+O(\|z-\bar x\|^2).
\end{align*}
In other words, we have verified that $\mathcal P$ satisfies the retraction property $(ii)$ on a neighborhood of $\bar x$. Consequently for all $x\in \cX$ near $\bar x$, we may set $z=\mathcal{P}^{-1}(x)$ and deduce
$x=\mathcal P^{-1}(x)+O(\|P^{-1}(x)-P^{-1}(\bar x)\|^2)=\mathcal P^{-1}(x)+O(\|x-\bar x\|^2).$ Therefore $\dist(x,\cX_{\bar x})=O(\|x-\bar x\|^2)$. Hence \eqref{eqn:need_this_eqn} holds around $\bar x$ for some $r_1,\tau_1>0$.	
\end{proof}

%	Beyond projections onto the manifold, 
%	any retraction 
%	in 
%	the sense 
%	of 
%	manifold optimization \cite{absil2009optimization} automatically 
%	satisfies  
%	$(ii)$ 
%	for 
%	some 
%	$\accR $ 
%	and 
%	$r_2$. 
%	For example, 
%	one 
%	may set 
%	$\cR_x$ 
%	to be 
%	the nearest point projection 
%	from 
%	 $x+T_{\cX}(x)$ 
%	onto 
%	$\cX$. 
	Notice that the subproblems
	solved by 
	Algorithm~\ref{alg:stoc_prox} 
	are 
	convex 
	for 
	sufficiently large $\beta_t$
	and 
	are 
	therefore globally solvable.
	In particular, 
	when equipped with the subgradient models 
	in 
	Table~\ref{table:tab}, 
	Algorithm~\ref{alg:stoc_prox} 
	becomes 
	a stochastic Riemannian subgradient method.

	\subsection{Functional constraints}
	In this example, 
		we 
			assume that 
		the constraint set $\cX$ 
	is 
	a sublevel set
	\begin{equation*}%\label{eqn:singleeqn}
	\cX=[g\leq 0],
	\end{equation*}
	where $g\colon\R^d\to\R\cup\{\infty\}$ 
	is 
	a closed function. 
Approximations 
	of 
	the set $\cX$ naturally 
	arise 
	from 
	approximations 
	of 
	the function $g$. 
	Namely 
	suppose 
	that 
	for 
	each 
	$x\in \cX$, 
	we 
	have
	a convex function $g_x\colon\R^d\to\R\cup\{\infty\}$ that 
	is 
	a two-sided model of $g$, 
	meaning 
	$$|g_x(y)-g(y)|\leq \frac{\gamma}{2}\|y-x\|^2\qquad \forall y\in \R^d,$$
	where $\gamma>0$ 
	is 
	a constant independent of $x$.
	Then we 
	may define 
	the set approximations
	$$\cX_x:=\left\{y: g_x(y)+\frac{\gamma}{2}\|y-x\|^2\leq 0\right\}.$$
	Typical examples 
	of 
	such two-sided models
	and 
	the induced set approximations 
	are 
	summarized 
	in 
	Table~\ref{table:tab} and Figure~\ref{fig:inequality}, illustrating 
	that one   
			may also model 
		multiple inequalities $g_i\leq 0$ for $i=1,\ldots,m$ 
			with 
			a single constraint function, namely $g=\max_{i=1,\ldots, m} g_i$. 
	Such approximations 
		have a computational benefit, 
			since they never require us to project onto $\cX$. 
	Indeed, 	
		since the inclusion $\cX_x\subseteq\cX$ 
		clearly holds, 
		the condition $(ii)$ holds 
			with 
		the identity retraction $\cR_x={\rm Id}$.

%
%	
%	
%	Although it may seem that this assumption limits us to sets defined by a single inequality, 
%		one   
%			may in fact model 
%		multiple inequalities $g_i\leq 0$ for $i=1,\ldots,m$ 
%			with 
%			a single constraint function, namely $g=\max_{i=1,\ldots, m} g_i$.  
%	
%	In this setting, 
%	it 
%	is 
%	important to avoid retractions, 
%	since they are often computationally expensive. 
%	We 
%	will now see 
%	how 
%	to form 
%	set-approximations that 
%	are 
%	conveniently contained in $\cX$, thereby avoiding the need for retractions.

%	
%	Approximations 
%	of 
%	the set $\cX$ naturally 
%	appear 
%	as sublevel sets 
%	of 
%	approximations 
%	of 
%	the function $g$. 
%	Namely 
%	suppose 
%	that 
%	for 
%	each 
%	$x\in \cX$, 
%	we 
%	may find 
%	a convex function $g_x\colon\R^d\to\R\cup\{\infty\}$ that 
%	is 
%	a two-sided model of $g$, 
%	meaning 
%	$$|g_x(y)-g(y)|\leq \frac{\gamma}{2}\|y-x\|^2\qquad \forall y\in \R^d,$$
%	where $\gamma>0$ 
%	is 
%	a constant independent of $x$.
%	Then we 
%	may define 
%	the set approximations
%	$$\cX_x:=\left\{y: g_x(y)+\frac{\gamma}{2}\|y-x\|^2\leq 0\right\}.$$
%	Typical examples 
%	of 
%	such two-sided models
%	and 
%	the induced set approximations 
%	are 
%	summarized 
%	in 
%	Table~\ref{table:tab}. 
%	See also Figure~\ref{fig:inequality} for an illustration.

	In the remainder of the section, we will verify condition $(i)$. 
	Our strategy 
		is 
	as follows. 
	First 
	consider 
	a point $x$ with $g(x)\ll 0$. 
	Then intuitively, one
	should be able to fit 
	a small neighborhood $\BB(x,r_1)\cap \cX$ completely inside $\cX_x$. Consequently \eqref{eqn:need_this_eqn} will hold with $\accD=0$. 
	This argument 
	breaks down 
	wherever $g(x)$ 
	is 
	arbitrarily close 
	to 
	zero, 
	since the radius $r_1$ 
	would have to become 
	arbitrarily small. 
	Therefore, one 
	has to consider 
	the boundary behavior separately. 
	To this end, as is standard when working with functional constraints, 
	we 
	will assume 
	a qualification condition \eqref{eqn:MFCQ}, 
	which essentially guarantees that the constraint region $[g\leq r]$ 
	behaves 
	in a Lipschitz manner relative to perturbations of the right side $r\approx 0$. In particular, when the constraint set $\cX$ is cut out by smooth nonlinear inequalities, property~\eqref{eqn:MFCQ} reduces to the classical Mangasarian–Fromovitz Constraint Qualification (MFCQ).

%	{\color{red} Dima: should we say that this is related to MFCQ and cite an appropriate paper?}
	
%	{\color{red} Dima: Should we add the sufficient condition on the models to the theorem? This may be interesting even for smooth functions $f(y)$. Since then the set approximation is a ball and I think the sharp growth condition is easy to check....}

	\begin{theorem}[Positive slope implies $\eqref{eqn:need_this_eqn}$]\label{thm:main_thm_func_constr}
		Suppose that there are constants $\kappa>0$ and $\alpha<0<\beta$ such that the condition holds:
\begin{equation}\label{eqn:MFCQ}
\dist(0, \partial g(x))\geq \frac{1}{\kappa}\qquad \forall x\in [\alpha<g\leq \beta].
\end{equation}
Suppose moreover that $g$ is $L$-Lipschitz continuous on the tube ${U}_{\cX}(\epsilon) \cap \dom(g)$ for some $\epsilon >0$. 	Choose a constant $$\delta<\min\left\{\frac{\epsilon}{3}, \frac{\beta}{2L},\frac{1}{14\gamma\kappa},\frac{L}{\gamma}\left(\sqrt{1-\frac{\alpha\gamma}{2L^2}}-1\right)\right\}.$$
Then the estimate $\eqref{eqn:need_this_eqn}$ holds with  
$\accD=2\gamma \kappa$ and $r_1=\min\left\{\frac{\delta}{2},\sqrt{\frac{\beta-2L\delta}{\gamma}},\sqrt{\frac{\delta}{8\kappa\gamma}} \right\}.$
	\end{theorem}
	
	The following section develops a proof of Theorem~\ref{thm:main_thm_func_constr}.	The argument uses standard variational analytic techniques based on error bounds, but is somewhat technical; the reader can safely skip it upon first reading.

	\subsubsection{Proof of Theorem~\ref{thm:main_thm_func_constr}}\label{sec:app_proof_level}
	Henceforth, let $r^+:=\max\{0,r\}$ denote the positive part of any real $r\in \R$. For each $x\in \R^d$, we define the function $$G_{ x}(y):=g_{ x}(y)+\frac{\gamma}{2}\|y- x\|^2.$$ Note the equality $\cX_x=[G_x\leq 0]$.	
	We begin with the following key lemma, which gives a minimal sufficient condition for establishing~\eqref{eqn:need_this_eqn}. It shows that a local error bound property, which asserts a Lipschitz-like behavior of sublevel sets $[G_{\bar x}(\cdot)\leq r]$ with respect to $r$, implies a local estimate \eqref{eqn:need_this_eqn} around a base point.

	\begin{lemma}[Error bound for models implies \eqref{eqn:need_this_eqn}]\label{lem:mod_error_give_result}
			Fix a point ${\bar x}\in \cX$ and suppose that there exist $\kappa,r,\epsilon>0$ such that the local error bound holds:
		\begin{equation}\label{eqn:loc_errodir}
		\dist(y,[G_{\bar x}\leq 0])\leq \kappa\cdot [G_{\bar x}(y)]^+,\qquad \forall y\in B({\bar x},\epsilon)\cap [G_{\bar x}\leq r].
		\end{equation}
		Then the inequality holds:
		$$\dist(y,\cX_{\bar x})\leq \gamma\kappa\cdot \|y- {\bar x}\|^2 \qquad \forall y\in \cX\cap \BB\left(\bar x, \delta\right),$$
		where we define the radius $\delta:= \min\{\epsilon,\sqrt{ r/\gamma}\}$.
	\end{lemma}
	\begin{proof}
	Consider a point $y\in \cX\cap\mathbb{B}(\bar x,\delta)$. Observe $$G_{\bar x}(y)=g_{\bar x}(y)+\frac{\gamma}{2}\|y-\bar x\|^2\leq g(y)+\gamma\|y-\bar x\|^2\leq \gamma\delta^2\leq r.$$
	Thus $y$ lies in $\mathbb{B}(\bar x,\epsilon)\cap [G_{\bar x}\leq r].$ Appealing to the assumption \eqref{eqn:loc_errodir}, we conclude 
	$$\dist(y,\cX_{\bar x})=\dist(y,[G_{\bar x}\leq 0])\leq \kappa \cdot [G_{\bar x}(y)]^{+}\leq \gamma\kappa\|y-\bar x\|^2.$$
This completes the proof.
	\end{proof}
	
	While it shows that sharp growth of the model function $G_{\bar x}$ implies~\eqref{eqn:need_this_eqn}, the deficiency of Lemma~\ref{lem:mod_error_give_result} is that the key error bound assumption \eqref{eqn:loc_errodir} is stated in terms of the models $G_{\bar x}$ and not in terms of original function $g$. Therefore, our goal is now to develop conditions on $g$ itself that ensure \eqref{eqn:loc_errodir}. To do this, we follow a two step argument. First, we show~\eqref{eqn:loc_errodir} follows whenever the subgradients of $G_{\bar x}$ have sufficiently large norm in a small shell outside $\cX_x$, a condition akin to~\eqref{eqn:MFCQ}. Then we show that~\eqref{eqn:MFCQ} implies the aforementioned condition on the subgradients of $G_{\bar x}$.

	Turning to the the first step,  we must lower bound $\dist(0, \partial G_{\bar x}(y))$ for $y$ near $\bar x$. This quantity in turn may equivalently be stated in terms of the $\emph{slope}$: for any function $h\colon\R^d\to\R\cup\{\infty\}$ and any point $x\in \dom h$, we define the slope
	$$|\nabla h|(x):=\limsup_{y\to x}\frac{(h(x)-h(y))^+}{\|x-y\|}.$$
	To see the relation, observe that if $h$ is weakly convex, then the slope reduces to
	$$|\nabla h|(x)=\dist(0,\partial h(x)),$$
	since directional derivative of $h$ is the support function of the subdifferential $\partial h$.
	The slope is well-known to be closely related with the error bound property. For example, the following lemma provides a slope-based sufficient condition for an error bound to hold at a single point.

	\begin{lemma}[Slope and one-point error bound{\cite[Basic Lemma, Ch. 1]{ioffe_survey}}]\label{lem:ioffe_desc}
	Consider a closed function $h\colon\R^d\to\R\cup\{\infty\}$. Fix a point $\bar x\in \dom h$ and assume there are constants $\alpha, \kappa, \epsilon>0$ satisfying $0\leq h(\bar x)-\alpha<\epsilon/\kappa$ and 
	such that 
	$$|\nabla h|(x)\geq \frac{1}{\kappa}\qquad \textrm{for all }x\in \mathbb{B}(\bar x,\epsilon)\cap [\alpha<h\leq h(\bar x)].$$ 
	Then the estimate, $\dist(\bar x,[h\leq \alpha])\leq \kappa(h(\bar x)-\alpha)$, holds.
	\end{lemma}
	
	From this one-point error bound, we can now easily establish a slope-based sufficient condition for an error bound to hold on a ball. We provide a short proof for completeness.

	\begin{corollary}[Slope and a local error bound]\label{cor:error_on_ball}
	Consider a closed function $h\colon\R^d\to\R\cup\{\infty\}$ and fix a point $\bar x\in \dom h$. Assume there are constants $\alpha,\beta, \kappa, \epsilon>0$ satisfying $0<\beta-\alpha<\frac{\epsilon}{2\kappa}$ and 
	such that 
	\begin{equation}\label{eqn:second_try}
	|\nabla h|(x)\geq \frac{1}{\kappa}\qquad \textrm{for all }x\in \mathbb{B}(\bar x,\epsilon)\cap [\alpha<h\leq \beta].
	\end{equation}
	Then the estimate
	$$\dist( x,[h\leq \alpha])\leq \kappa(h( x)-\alpha)^+\qquad \textrm{holds for all }x\in \mathbb{B}(\bar x,\epsilon/2)\cap [h\leq \beta].$$
\end{corollary}
	\begin{proof}
	Observe first that for any $ x\in [h\leq \alpha]$, the estimate $\dist( x,[h\leq \alpha])\leq \kappa(h( x)-\alpha)^+$ holds trivially.
	Therefore, we may only focus on an arbitrary point $y\in \mathbb{B}(\bar x,\epsilon/2)\cap [\alpha<h\leq \beta]$. Let us verify the assumptions of Lemma~\ref{lem:ioffe_desc} with $\bar x$ replaced by $y$ and $\epsilon$ replaced by $\epsilon/2$. To this end, consider an arbitrary point $x\in \mathbb{B}(y,\epsilon/2)\cap[\alpha<h\leq h(y)]$. Then clearly $x$ lies in $\mathbb{B}(\bar x,\epsilon)\cap [\alpha<h\leq \beta]$, and therefore $|\nabla h|(x)\geq \frac{1}{\kappa}$ holds by \eqref{eqn:second_try}. Moreover, observe $0<h(y)-\alpha\leq\beta-\alpha<\frac{\epsilon/2}{\kappa}$.
	 An application of Lemma~\ref{lem:ioffe_desc} therefore guarantees $\dist(y,[h\leq \alpha])\leq \kappa(h(y)-\alpha)^+$. The proof is complete.
	\end{proof}

	Taking Corollary~\ref{cor:error_on_ball}, together with Lemma~\ref{lem:mod_error_give_result} yields sufficient conditions for \eqref{eqn:need_this_eqn} based on the subdifferential of $G_{\bar x}(y)$.

	\begin{corollary}[Nondegenerate slope for the models implies \eqref{eqn:need_this_eqn}]\label{cor:nondeg_slop_model}
				Fix a point ${\bar x}\in \cX$. Assume there are constants $r, \kappa, \epsilon>0$ satisfying $0<r<\frac{\epsilon}{2\kappa}$ and 
				such that 
				\begin{equation}\label{eqn:get_it}
				|\nabla G_{\bar x}|(y)\geq \frac{1}{\kappa}\qquad \textrm{for all }y\in \mathbb{B}(\bar x,\epsilon)\cap [0<G_{\bar x}\leq r].
				\end{equation}
						Then the inequality holds:
				$$\dist(y,\cX_{\bar x})\leq \gamma\kappa\cdot \|y- {\bar x}\|^2 \qquad \forall y\in \cX\cap \BB\left(\bar x, \nu\right),$$
				where we define the radius $\nu:= \min\{\epsilon/2,\sqrt{ r/\gamma}\}$.
	\end{corollary}
%	\begin{proof}
%			Then the estimate
%		$$\dist( x,[G_{\bar x}\leq 0])\leq \kappa(G_{\bar x}( x))^+\qquad \textrm{holds for all }x\in \mathbb{B}(\bar x,\epsilon/2)\cap [G_{\bar x}\leq \beta].$$
%		
%		\end{proof}

	Turning to the second step of our argument, we now determine conditions on the slope of $g$ that imply the lower bound \eqref{eqn:get_it} on the slope  for the model $G_{\bar x}(\cdot)$. To this end, we will require the following result (a special case of {\cite[Theorem 6.1]{davis2017nonsmooth}}) that compares the subdifferentials of $g$ and $G_{\bar x}$.

	\begin{lemma}[Slope comparison{\cite[Theorem 6.1]{davis2017nonsmooth}}]\label{lem:subgrad_closeness}
	Fix a point $\bar x\in \R^d$ and an arbitrary constant $\lambda>0$. Then for every point $y_1\in\R^d$ and subgradient $v_1\in\partial G_{\bar x}(y_1)$, there exists another point $y_2\in\R^d$ satisfying 
	$$\|y_1-y_2\|\leq 2\lambda\qquad\textrm{and}\qquad\dist(v_1,\partial g(y_2))\leq \frac{\gamma\|y_1-\bar x\|^2}{\lambda}+2\gamma\|y_2-\bar x\|.$$
%	and
%	$$g_2(y_2)\leq g_x(y_1)+\|v_1\|\|y_2-y_1\|+\frac{\gamma}{2}\|y_2-x\|^2$$
	\end{lemma}
%	\begin{proof}
%	This follows immediately from \cite{}, applied to the pair of functions $g(\cdot)$ and $g_x(\cdot)$ and setting $-u(\cdot)=\ell(\cdot)=\frac{\gamma}{2}\|\cdot-x\|^2$. 	Then get 
%	$\|y_1-y_2\|\leq 2\lambda$ and 	$$\dist(v_1,\partial g(y_2))\leq \lambda^{-1}\gamma\|y_1-x\|^2+\gamma\|y_2-x\|.$$ Triangle inequality gives
%	$$\dist(v_1,\partial g(y_2))\leq \lambda^{-1}\gamma\|y_1-x\|^2+\gamma\|y_2-y_1\|+\gamma\|y_1-x\|=  \gamma(\lambda^{-1}\|y_1-x\|^2+\|y_1-x\|) +2\lambda\gamma.$$
%	\end{proof}

Using Lemma~\ref{lem:subgrad_closeness}, we can pass from a lower bounds on $\dist(0,\partial g(y))$ to a lower bound on $\dist(0,\partial G_{\bar x}(y))$.

	\begin{corollary}[Slope of the objective and the models]\label{cor:main_supercor}
		Fix an arbitrary point $\bar x\in\R^d$ and suppose that there are constants $\epsilon,\kappa>0$ and $\alpha<0<\beta$ such that the condition holds:
	\begin{equation}\label{eqn:low_subg}
	\dist(0, \partial g(x))\geq \frac{1}{\kappa}\qquad \textrm{for all }x\in \mathbb{B}(\bar x,\epsilon)\cap[\alpha<g\leq \beta].
	\end{equation}
	Suppose that $g$ is $L$-Lipschitz constinuous on the ball $\mathbb{B}(\bar x,\epsilon) \cap \dom(g)$. 
	Choose a constant $$\delta<\min\left\{\frac{\epsilon}{3}, \frac{\beta}{2L},\frac{1}{14\gamma\kappa},\frac{L}{\gamma}\left(\sqrt{1-\frac{\alpha\gamma}{L^2}}-1\right)\right\}.$$ Then the estimate holds:
	$$\dist(0,\partial G_{\bar x}(y))\geq \frac{1}{2\kappa}\qquad \textrm{for all }y\in \mathbb{B}(\bar x,\delta)\cap [\ell<G_{\bar x}\leq u],$$
	where the constants $u:=\beta-2L\delta$ and $\ell:=\alpha+\gamma\delta^2+2L\delta$ satisfy $\ell<0<u$.
	\end{corollary}

	\begin{proof}
	The inequalities $\ell<0<u$	are immediate from the definition of $\delta$. Fix now a point $y_1\in \mathbb{B}(\bar x,\delta)\cap [\ell<G_{\bar x}\leq u]$ and a subgradient $v_1\in \partial G_{\bar x}(y_1)$ of minimal norm. Applying Lemma~\ref{lem:subgrad_closeness} with $\lambda:=\delta$ guarantees that there exists a point $y_2\in \R^d$ satisfying 
		$$\|y_1-y_2\|\leq 2\delta\qquad\textrm{and}\qquad\dist(v_1,\partial g(y_2))\leq \frac{\gamma\|y_1-\bar x\|^2}{\delta}+2\gamma\|y_2-\bar x\|.$$
	In particular, we deduce 
	$$\|y_2-\bar x\|\leq \|y_1-y_2\|+\|y_1-\bar x\|\leq 3\delta< \epsilon.$$
	It follows that $y_1$ and $y_2$ both lie in $\mathbb{B}(\bar x,\epsilon)$. Using Lipschitz continuity of $g$ on this ball, we deduce
		$$g(y_2)\leq g(y_1)+L\|y_1-y_2\|\leq u +2L\delta = \beta,$$
	and 
	$$g(y_2)\geq g(y_1)-L\|y_1-y_2\|\geq G_{\bar x}(y_1)-\gamma\|y_1-\bar x\|^2-L\|y_1-y_2\|>\ell-\gamma\delta^2-2L\delta= \alpha.$$
	 Therefore, the inclusion $y_2\in \mathbb{B}(\bar x,\epsilon)\cap [\alpha<g\leq \beta]$ holds. Using \eqref{eqn:low_subg}, we deduce 
	 \begin{align*}
	 \kappa^{-1}\leq \dist(0,\partial g(y_2))&\leq \|v_1\|+\dist(v_1,\partial g(y_2))\\
	 &\leq \|v_1\|+\frac{\gamma\|y_1-\bar x\|^2}{\delta}+2\gamma\|y_2-\bar x\|\\
	 &\leq \|v_1\|+7\gamma\delta.
	 \end{align*}
	 Rearranging yields,
	 $\dist(0,\partial G_{\bar x}(y_1))=\|v_1\|\geq \kappa^{-1}-7\gamma\delta> \frac{1}{2\kappa},$
	 as claimed.	 
%	$\|y_1-y_2\|\leq 2\lambda$ and $\|y_2-x\|\leq \|y_1-x\|+\|y_1-y_2\|\leq \|y_1-x\|+2\lambda$ and
%	\begin{align*}
%	\dist(0,\partial g(y_2))&\leq \|v_1\|+\dist(v_1,\partial g(y_2))\\
%	&\leq \|v_1\|+\gamma(\lambda^{-1}\|y_1-x\|^2+\|y_1-x\|) +2\lambda\gamma.
%	\end{align*}	
%	and 
%	$$g(y_2)\leq g(y_1)+L\|y_1-y_2\|\leq r'+2L\lambda.$$
	\end{proof}

We are now ready for the proof of Theorem~\ref{thm:main_thm_func_constr}.

\begin{proof}[Proof of Theorem~\ref{thm:main_thm_func_constr}]
	As in Corollary~\ref{cor:main_supercor}, define $u:=\min\{\beta-2L\delta, \frac{\delta}{8\kappa}\}$ and $\ell:=\alpha+\gamma\delta^2+2L\delta$ and note $\ell<0<u$.
Fix an arbitrary point $\bar x\in [\ell< g\leq 0]$. Then Corollary~\ref{cor:main_supercor} yields the estimate:
	$$\dist(0,\partial G_{\bar x}(y))\geq \frac{1}{2\kappa}\qquad \textrm{for all }y\in \mathbb{B}(\bar x,\delta)\cap [\ell<G_{\bar x}\leq u],$$
Taking into account $0<u<\frac{\delta}{4\kappa}$, Corollary~\ref{cor:nondeg_slop_model} immediately guarantees
		$$\dist(y,\cX_{\bar x})\leq 2\gamma\kappa\cdot \|y- {\bar x}\|^2 \qquad \forall y\in \cX\cap \BB\left(\bar x, r_1\right).$$	
		Next, fix a point $\bar x\in [g\leq \ell]$. Then for every point $y\in \cX\cap \BB(\bar x,r_1)$, we compute
		\begin{align*}
		G_{\bar x}(y)\leq g(y)+\gamma\|y-\bar x\|^2
		&\leq  g(\bar x)+L\|y-\bar x\|+\gamma\|y-\bar x\|^2\\
		&\leq \ell+Lr_1+\gamma r_1^2\\
		&\leq \alpha+\gamma\delta^2+2L\delta+L\delta+\gamma\delta^2\leq 0.
		\end{align*}
		We conclude $\dist(y,\cX_{\bar{x}})=0$.	
		The proof is complete.
%where we define the radius $\nu:= \min\{\delta/2,\sqrt{ u/\gamma}\}$. The proof is complete.
%Consider now a point $x\in [g\leq \ell]$. Then for every point $y\in \cX\cap \BB(x,r_1)$, we compute
%$$G_x(y)\leq g(y)+\gamma\|y-x\|^2\leq  \bar \ell+L\|y-x\|+\gamma\|y-x\|^2\textcolor{red}{\leq 0}.$$
%We conclude $\dist(y,\cX_x)=0$,	
\end{proof}

	\section{Retracted model-based algorithm}\label{sec:main_results}
	In this section, 
	we 
	generalize 
	Algorithm~\ref{alg:stoc_prox_noret} 
	by 
	allowing one 
	to replace 
	$\cX$ 
	in 
	the subproblem \eqref{eqn:subprob_to_solve} by a close approximation.
	Namely, 
	the algorithm we propose (Algorithm~\ref{alg:stoc_prox}) 
	will use 
	three building blocks: 
	a family 
	of 
	stochastic models $f_{x}(\cdot,\xi)$ 
	of 
	the objective function, 
	a family 
	of 
	proximally smooth sets $\cX_x$ 
	that 
	approximate $\cX$ near $x$, 
	and 
	a retraction operation $\cR_x\colon\cX_x\to\cX$ 
	that 
	restores feasibility. 
	Thus in each iteration $t$, 
	Algorithm~\ref{alg:stoc_prox} 
	forms 
	both 
	a proximally smooth approximation $\cX_{x_t}$ 
	of 
	the original constraint 
	and  
	a stochastic model $f_{x_t}(\cdot,\xi_t)$ 
	of 
	the objective function, 
	centered at the current iterate $x_t$. 
	The procedure then 
	computes 
	a minimizer ${\tilde x}_t$ 
	of 
	the function  $f_{x_t}(\cdot,\xi_t)+\frac{\beta}{2}\|\cdot-x_t\|^2$ 
	over
	$\cX_{x_t}$ 
	and 
	retracts 
	it back to $\cX$, thereby defining the next iterate $x_{t+1}=\cR_{x_t}(\tilde x_t)$.

	\smallskip
	\begin{algorithm}[H]
		{\bf Input:} initialization $x_0\in \R^d$,   a sequence $\beta_t>0$,  and an iteration count $T\in \mathbb{N}$.
		\break{\bf Step } $t=0,\ldots,T$:	
		\begin{equation*}
		\begin{aligned}
		&\textrm{Sample } \xi_t \sim P\\
		& \textrm{Choose } \tilde x_{t} \in \argmin_{x \in \cX_x}~ \left\{f_{x_t}(x,\xi_t) + \frac{\beta_t}{2} \|x - x_t\|^2\right\}\\
		& \textrm{Set }  x_{t+1} = \cR_{x_t}(\tilde x_t)
		\end{aligned}
		\end{equation*}
		\caption{Retracted Stochastic Model Based Algorithm
			%: PSSM($x_0,T,\{\alpha_t\}$)
		}
		\label{alg:stoc_prox}
	\end{algorithm}
	\smallskip

	The success of Algorithm~\ref{alg:stoc_prox}, unsurprisingly, relies on the approximation quality of both the stochastic models $f_{x_t}$ and the set approximations $\cX_{x_t}$. Henceforth, we impose the following assumptions.
	
	\begin{assumption}
		{\rm
			Fix a probability space $(\Omega,\mathcal{F},P)$ and equip $\R^d$ with the Borel $\sigma$-algebra. 
			We assume that there exist real $\eta,\osa, L, R, \accD, \accR, r_1, r_2 \in \R$ satisfying the following properties. 
			\begin{enumerate}
				\item[(B1)] {\bf (Sampling)} It is possible to generate i.i.d.\ realizations $\xi_1,\xi_2, \ldots \sim P$.
				\item[(B2)]\label{it:B2} {\bf (One-sided accuracy)} There is an open set $U$ containing $\cX $ and a measurable function $(x,y,\xi)\mapsto f_x(y,\xi)$, defined on $U\times U\times\Omega$, satisfying  $$\EE_{\xi}\left[f_x(x,\xi)\right]=f(x)
				\quad \textrm{and} 
				\quad\EE_{\xi}\left[f_x(y,\xi)-f(y)\right]
				\leq\frac{\osa}{2}\|y-x\|^2\qquad \forall x,y\in U.$$
				\item[(B3)] {\bf (Weak-convexity)} The function $f_x(\cdot,\xi)$ is $\eta$-weakly convex $\forall x\in U$, a.e. $\xi\sim P$.
				\item [(B4)] {\bf (Set approximation)} There exists a collection $\{(\cX_x,\cR_{x})\}_{x\in\cX}$ that is a set approximation of $\cX$ with parameters $(R,\accD,r_1,\accR ,r_2)$.
				\item[(B6)] {\bf (Lipschitz property)} $f$ is $L$-Lipschitz continuous on $\cX$; and for all $x\in \cX$ and a.e. $\xi\sim P$, the function $f_x(\cdot, \xi)$ is $L$-Lipschitz continuous on some neighborhood of $\cX_x$.
		\end{enumerate}}
	\end{assumption}

	The following theorem---the main result of this work---summarizes convergence guarantees for Algorithm~\ref{alg:stoc_prox}.

	\begin{theorem}[Convergence guarantees]\label{thm:gen_conv_guarant}
		Without loss of generality, suppose $r_1<R$ and define the constants $\nu := \frac{R}{2(R-r_1)^2} $ and $\gamma:=\eta+3L\nu$. Fix a real $\bar{\rho} > \max\{\frac{2L}{r_1}, \gamma+\osa+3\accD L\}$ and a sequence $\beta_t > \max\{\frac{2L}{r_2},\gamma\}$. Let $\{x_t\}_{t=0}^T$ be the iterates generated by  Algorithm~\ref{alg:stoc_prox} and let  $t^*\in \{0,\ldots,T\}$ be sampled according to the discrete probability distribution
		$\mathbb{P}(t^*=t)\propto\frac{1}{\beta_t - \eta-\frac{3LR}{2(R-r_1)^2}}$.
		Then $x_{t^*}$ satisfies the estimate:	
		\begin{align*}
		\EE_t\left[\cC_{1/\bar \rho}(x_{t^\ast})^2\right] \leq \frac{2\bar\rho(f(x_0) - \min_{\cX} f) + \bar\rho^2L^2\sum_{t=0}^T a_t}{\sum_{t=0}^T \frac{\bar \rho - \gamma - \osa -  3\accD L}{\beta_t - \gamma }},
		\end{align*}
		where we define
		$$
		a_t :=  \frac{1}{\beta_t(\beta_t - \gamma)} + \frac{8\accR L(\bar \rho^{-1} + \beta_t^{-1})}{\beta_t^2} + \frac{4\accR ^2 L^2}{\beta_t^4}.
		$$
		In particular, if we set $\beta_t=\gamma+\frac{\sqrt{T+1}}{\alpha}$ for some positive  $\alpha < \frac{r_2}{2L-\gamma r_2}$, then it holds:
		\begin{align*}
		\EE_t\left[\cC_{1/\bar \rho}(x_{t^\ast})^2\right] \leq & \frac{2\bar\rho(f(x_0) - \min_{\cX} f)}{\alpha(\bar \rho - \gamma - \osa -  3\accD L)\sqrt{T+1}}\\
		&\quad+\frac{\alpha\bar \rho^2L^2}{\bar \rho - \gamma - \osa -  3\accD L}\left[\frac{1+\accR L(\bar \rho^{-1}+\frac{\alpha}{\sqrt{T+1}})}{\sqrt{T+1}}+\frac{4\alpha^2\accR ^2L^2}{(T+1)^{3/2}}\right].
		\end{align*}
		
	\end{theorem}
	
	The following section develops a proof of Theorem~\ref{thm:gen_conv_guarant}.
	
	\subsection{Proof of Theorem~\ref{thm:gen_conv_guarant}}

	The proof parallels that of Theorem~\ref{thm:simple_result}. The main technical difficulty is that we will need to modify the three point inequality slightly in order to take into account the presence of retractions. 
	In what follows, we let   $x_0,\ldots, x_t$ and $\xi_0,\ldots, \xi_{t}$ be the iterates and random elements generated by Algorithm~\ref{alg:stoc_prox}. We will use the shorthand $\EE_t$ to denote the expectation conditioned on $\xi_0,\ldots, \xi_{t-1}$.  We begin with a key one-step improvement lemma. 
	
	The main part of the analysis is an analogue of the one-step decrease Lemma~\ref{lem:one_step}; this is the content of the following lemma. Throughout, without loss of generality, we assume $r_1\leq R$.

	\begin{lemma}[One-step improvement]\label{lem:stoc_prox_descent}
		Fix an index $t$ and choose  $\hat x_t\in \cP_{1/\bar \rho}(x)$. Then the estimate holds:
		\begin{equation}
		\EE_t\left[\|\hat x_t - x_{t+1}\|^2\right] \leq  \frac{\beta +\osa+ 3\accD L - \bar \rho }{\beta- \gamma}\|x_t - {\hat x}_t \|^2 + L^2 a_t.\label{ineq:stoc_prox_descent}
		\end{equation}
		%	\end{align}
	\end{lemma}
	\begin{proof}
		Since the statement of the theorem is independent of $x_0,\ldots, x_{t-1}$, we will simplify the notation by dropping the index $t$ in $x_t$, ${\tilde x}_t$, ${\hat x}_t$, $\xi_t$, $\EE_t$, $a_t$ and setting $x^+:=x_{t+1}$.
		By the definition of $\tilde x$, there exist vectors $v \in \partial f_{x}(\tilde x, \xi)$ and $w \in N_{\cX_{x}}(\tilde x)$ satisfying  $\beta(x - \tilde x) = v + w.$
		Using Lemma~\ref{lem:prox_point_bound} and Lipschitz continuity of $f_x(\cdot,\xi)$ yields the estimate $\|w\| \leq 3L$. Fix any  $y \in \cX\cap\BB(x,r_1)$. We then deduce
		\begin{align}
		\dotp{w, y- \tilde x} &= \dotp{w, (y - \proj_{\cX_x}(y))} + \dotp{w, \proj_{\cX_x}(y) - \tilde x}\notag\\
		&\leq \|w\|\cdot \dist(y, \cX_x) + \frac{\|w\|}{2R} \cdot\|\proj_{\cX_x}(y) - \tilde x\|^2\label{eqn:part2it1}\\
		&\leq \frac{\|w\|\accD}{2} \|y - x\|^2 + \frac{\|w\| \nu}{2} \|y - \tilde x\|^2,\label{eqn:part2it2}
		\end{align}
		where \eqref{eqn:part2it1} follows from the Cauchy-Schwarz inequality and from Lemma~\ref{lem_basic_prop_man}, while \eqref{eqn:part2it2} follows from (B4) and Lemma~\ref{lem_basic_prop_man}.
		Adding this estimate to the subgradient inequality for $f_x(\cdot,\xi)$ and completing the square, we conclude  for all $y \in \cX\cap\BB(x,r_1)$ the  bound:
		\begin{align}
		f_{x}(y, \xi) &\geq f_{x}(\tilde x, \xi) + \dotp{v + w, y- \tilde x} - \frac{\eta + \|w\| \nu}{2} \|y - \tilde x\|^2 - \frac{\|w\|\accD}{2}\|y - x\|^2\notag\\
		&\geq f_{x}(\tilde x, \xi) + \dotp{\beta(x - \tilde x) , y- \tilde x} - \frac{\eta + \|w\| \nu}{2} \|y - \tilde x\|^2  - \frac{\|w\|\accD}{2}\|y - x\|^2 \notag\\
		&= f_{x}(\tilde x, \xi) + \frac{\beta}{2}\left[ \|x - \tilde x\|^2 + \| y - \tilde x\|^2 - \|x - y\|^2\right]\label{eqn:perturb_three_point} \\
		&\qquad- \frac{\eta + \|w\| \nu}{2} \|y - \tilde x\|^2  - \frac{\|w\|\accD}{2}\|y - x\|^2.\notag
		\end{align}
		Lemma~\ref{lem:prox_point_bound} implies
		$
		\|\hat x - x\| \leq \frac{2L}{\bar \rho} \leq r_1.
		$
		Hence, setting $y = \hat x$ in \eqref{eqn:perturb_three_point} yields: 
		\begin{align}
		&\EE\left[\frac{\beta - \gamma}{2} \|\hat x - \tilde x\|^2 + \frac{\beta}{2}\| x - \tilde x\|^2 - \frac{\beta + \|w\|\accD }{2}\|x - \hat x\|^2\right]\notag\\
		&\leq \EE\left[f_x(\hat x, \xi) - f_x(\tilde x, \xi)\right]\notag \\
		&\leq \EE\left[f_x(\hat x, \xi) - f_x(x,\xi)\right]  + L \EE\left[\|x - \tilde x\|\right]\label{eqn:blah1}\\	
		&\leq f(\hat x) - f(x)  + L \EE\left[\|x - \tilde x\|\right] + \frac{\osa}{2}\|x - \hat x \|^2\label{eqn:blah2}\\
		&\leq \frac{\osa - \bar \rho }{2}\|x - \hat x \|^2 + L\sqrt{\EE\left[ \|x - \tilde x\|^2\right]},\label{eqn:blah3}
		%&\leq \frac{\tau - \bar \rho }{2}\|x - \hat x \|^2 + \frac{2L^2}{\beta}.
		\end{align}
		where \eqref{eqn:blah1} uses (B6), the estimate \eqref{eqn:blah2} uses (B2), and \eqref{eqn:blah3} uses the definition of $\hat x$ as the proximal point.
		Rearranging and setting $\delta :=\sqrt{ \EE\left[\|x - \tilde x\|^2\right]}$ yields the bound:
		\begin{equation}\label{eqn:intermed}
		\begin{aligned}
		\frac{\beta - \gamma}{2}\EE\left[\|\hat x - \tilde x\|^2\right] &\leq \frac{\beta + \accD\|w\|+ \osa - \bar \rho  }{2}\|x - \hat x \|^2 + L \delta - \frac{\beta }{2}\delta^2 \\
		&\leq \frac{\beta + \accD\|w\|+ \osa - \bar \rho  }{2}\|x - \hat x \|^2 + \frac{L^2}{2\beta},
		\end{aligned}
		\end{equation}
		where the last inequality follows by maximizing the expression  $L\delta - \frac{\beta }{2}\delta^2$ over $\delta\in \R$.
		To complete the proof, we will use the following estimate.
		\begin{claim}\label{claim1} The estimate holds:
			$
			\|\hat x -  x_+\|^2 \leq \|\hat x - \tilde x\|^2 + \frac{8\accR  L^3}{\beta^2}(\bar \rho^{-1} + \beta^{-1}) + \frac{4\accR ^2 L^4}{\beta^4}
			$.
		\end{claim}
		\begin{proof}
			Using Lemma~\ref{lem:prox_point_bound}, we deduce
			$\|\tilde x - x\| \leq 2L/\beta  \leq r_2$.
			In addition, we have $$\|\tilde x - x_+ \| = \|\tilde x - \cR_x(\tilde x)\| \leq \frac{\accR }{2}\|\tilde x - x\|^2\leq \frac{2\accR  L^2}{\beta^2}.$$ 
			Using the triangle inequality, we  deduce 
			$
			\|\hat x - x_+\|  \leq \|\hat x - \tilde x \| + \frac{2\accR  L^2}{\beta^2}.
			$
			Squaring both sides, we get
			$$
			\|\hat x - x_+\|^2  \leq \|\hat x - \tilde x \|^2 +  \frac{4\accR  L^2}{\beta^2}\|\hat x - \tilde x \| + \frac{4\accR ^2 L^4}{\beta^4}.
			$$
			Taking into account the estimate $
			\|\hat x - \tilde x \| \leq \|\hat x - x \| + \|x - \tilde x\| \leq 2L/\bar \rho + 2L/\beta,
			$
			completes the proof.
		\end{proof}
		Combining Claim~\ref{claim1} with the estimate \eqref{eqn:intermed}, we compute
		\begin{align*}
		\EE\left[\|\hat x -  x_+\|^2\right] &\leq \EE\left[\|\hat x - \tilde x\|^2\right] + \frac{8\accR  L^3(\bar \rho^{-1} + \beta^{-1})}{\beta^2} + \frac{4\accR ^2 L^4}{\beta^4} \\
		&\leq \frac{\beta + \accD\|w\|+ \osa - \bar \rho }{\beta- \gamma}\|x - \hat x \|^2 + L^2 a
		%\frac{L^2}{\beta(\beta - \gamma)} + \frac{8\accR  L^3(\bar \rho^{-1} + \beta^{-1})}{\beta^2} + \frac{4\accR ^2 L^4}{\beta^4}
		%	&=  \|x - \hat x \|^2  - \frac{\bar \rho - \gamma - \osa  -  \accD\|w\|}{\beta - \gamma}\|x - \hat x \|^2  + \frac{L^2}{\beta(\beta - \gamma)} + \frac{8\accR  L^3(\bar \rho^{-1} + \beta^{-1})}{\beta^2} + \frac{4\accR ^2 L^4}{\beta^4}.
		\end{align*}
		Bounding $\|w \| \leq 3L$ completes the proof of \eqref{ineq:stoc_prox_descent}.
	\end{proof}
	
	The convergence guarantees now quickly follow.

	\begin{proof}[Proof of Theorem~\ref{thm:gen_conv_guarant}]
		Fix an iteration $t$ and a  point  $\hat x_t\in\cP_{1/\bar \rho}(x_t)$. Then
		using the definition of the Moreau envelope and appealing to  Lemma~\ref{lem:stoc_prox_descent}, we deduce  
		\begin{align*}
		\EE_t[ \cM_{1/\bar \rho}(x_{t+1})] &\leq  \EE_t\left[f( \hat x_{t})  +\frac{\bar \rho}{2}\cdot \| x_{t+1} - \hat x_t \|^2\right],\\
		&\leq f( \hat x_{t}) +\frac{\bar \rho}{2}\left[\|\hat x_t-x_{t}\|^2+\left( \tfrac{\beta +\osa+ 3\accD L - \bar \rho }{\beta- \gamma}-1\right)\|\hat x_t-x_{t}\|^2  +L^2 a_t\right]\\
		&= \cM_{1/\bar \rho}(x_{t}) - \frac{\bar \rho- \gamma-\osa-3\accD L}{2\bar \rho(\beta_t-\gamma)}\cC_{1/\bar \rho}(x_t)^2  + \frac{\bar\rho L^2a_t}{2}.
		\end{align*}
		Taking expectations, iterating the inequality, and using the tower rule  yields:
		$$\sum_{t=0}^T \frac{\bar \rho - \gamma - \osa-3\accD L}{\beta_t-\gamma}\EE[\cC_{1/\bar{\rho}}(x_t)^2]\leq 2\bar{\rho}(\cM_{1/\bar \rho}(x_{0})-\min \cM_{1/\bar\rho})+  \bar \rho^2 L^2\cdot\sum_{t=0}^T a_t.$$
		Dividing through by $\sum_{t=0}^T \frac{\bar \rho - \gamma - \osa-3\accD L}{\beta_t-\gamma}$ and recognizing the left side as
		$\EE[\cC_{1/\bar{\rho}}(x_{t^*})^2]$ completes the proof.
	\end{proof}

	\section*{Conclusion}
	In this work, we presented a wide class of algorithms for minimizing weakly convex functions over proximally smooth sets and proved finite sample efficiency guarantees. The developed procedure allows one to mix approximations of both the objective function and the constraints within each iteration. We discussed consequences for  stochastic nonsmooth optimization over Riemannian manifolds (leading to Riemannian analogues of stochastic subgradient, proximal point, and prox-linear algorithms) and over sets cut out by nonlinear inequalities.

\bibliographystyle{siamplain}
\bibliography{bibliography}

\def\cfac#1{\ifmmode\setbox7\hbox{$\accent"5E#1$}\else
  \setbox7\hbox{\accent"5E#1}\penalty 10000\relax\fi\raise 1\ht7
  \hbox{\lower1.15ex\hbox to 1\wd7{\hss\accent"13\hss}}\penalty 10000
  \hskip-1\wd7\penalty 10000\box7}
\begin{thebibliography}{10}

\bibitem{absil2019collection}
{\sc P.-A. Absil and S.~Hosseini}, {\em A collection of nonsmooth riemannian
  optimization problems}, in Nonsmooth Optimization and Its Applications,
  Springer, 2019, pp.~1--15.

\bibitem{absil2009optimization}
{\sc P.-A. Absil, R.~Mahony, and R.~Sepulchre}, {\em Optimization algorithms on
  matrix manifolds}, Princeton University Press, 2009.

\bibitem{adly2016preservation}
{\sc S.~Adly, F.~Nacry, and L.~Thibault}, {\em Preservation of prox-regularity
  of sets with applications to constrained optimization}, SIAM Journal on
  Optimization, 26 (2016), pp.~448--473.

\bibitem{asi2}
{\sc H.~Asi and J.~C. Duchi}, {\em The importance of better models in
  stochastic optimization}, arXiv preprint arXiv:1903.08619,  (2019).

\bibitem{asi1}
{\sc H.~Asi and J.~C. Duchi}, {\em Stochastic (approximate) proximal point
  methods: Convergence, optimality, and adaptivity}, SIAM Journal on
  Optimization, 29 (2019), pp.~2257--2290.

\bibitem{AC}
{\sc D.~Az\'{e} and J.-N. Corvellec}, {\em Characterizations of error bounds
  for lower semicontinuous functions on metric spaces}, ESAIM Control Optim.
  Calc. Var., 10 (2004), pp.~409--425.

\bibitem{bacak2016second}
{\sc M.~Bac{\'a}k, R.~Bergmann, G.~Steidl, and A.~Weinmann}, {\em A second
  order nonsmooth variational model for restoring manifold-valued images}, SIAM
  Journal on Scientific Computing, 38 (2016), pp.~A567--A597.

\bibitem{balashov2019gradient}
{\sc M.~Balashov, B.~Polyak, and A.~Tremba}, {\em Gradient projection and
  conditional gradient methods for constrained nonconvex minimization}, arXiv
  preprint arXiv:1906.11580,  (2019).

\bibitem{balashov2019error}
{\sc M.~Balashov and A.~Tremba}, {\em Error bound conditions and convergence of
  optimization methods on smooth and proximally smooth manifolds}, arXiv
  preprint arXiv:1912.04660,  (2019).

\bibitem{beck}
{\sc A.~Beck and M.~Teboulle}, {\em A fast iterative shrinkage-thresholding
  algorithm for linear inverse problems}, SIAM J. Imaging Sci., 2 (2009),
  pp.~183--202.

\bibitem{boob2019proximal}
{\sc D.~Boob, Q.~Deng, and G.~Lan}, {\em Stochastic first-order methods for
  convex and nonconvex functional constrained optimization}, arXiv preprint
  arXiv:1908.02734,  (2019).

\bibitem{borckmans2014riemannian}
{\sc P.~B. Borckmans, S.~E. Selvan, N.~Boumal, and P.-A. Absil}, {\em A
  riemannian subgradient algorithm for economic dispatch with valve-point
  effect}, Journal of Computational and Applied Mathematics, 255 (2014),
  pp.~848--866.

\bibitem{Cartis2014}
{\sc C.~Cartis, N.~I. Gould, and P.~L. Toint}, {\em On the complexity of
  finding first-order critical points in constrained nonlinear optimization},
  Mathematical Programming, 144 (2014), pp.~93--106.

\bibitem{charisopoulos2019low}
{\sc V.~Charisopoulos, Y.~Chen, D.~Davis, M.~D{\'\i}az, L.~Ding, and
  D.~Drusvyatskiy}, {\em Low-rank matrix recovery with composite optimization:
  good conditioning and rapid convergence}, arXiv preprint arXiv:1904.10020,
  (2019).

\bibitem{charisopoulos2019composite}
{\sc V.~Charisopoulos, D.~Davis, M.~D{\'\i}az, and D.~Drusvyatskiy}, {\em
  Composite optimization for robust blind deconvolution}, arXiv preprint
  arXiv:1901.01624,  (2019).

\bibitem{chen1993convergence}
{\sc G.~Chen and M.~Teboulle}, {\em Convergence analysis of a proximal-like
  minimization algorithm using bregman functions}, SIAM Journal on
  Optimization, 3 (1993), pp.~538--543.

\bibitem{chen2019proximal}
{\sc S.~Chen, S.~Ma, A.~Man-Cho~So, and T.~Zhang}, {\em Proximal gradient
  method for nonsmooth optimization over the stiefel manifold}, SIAM Journal on
  Optimization, 30 (2020), pp.~210--239.

\bibitem{clarkelower}
{\sc F.~Clarke, R.~Stern, and P.~Wolenski}, {\em Proximal smoothness and the
  lower-{$C^2$} property}, J. Convex Anal., 2 (1995), pp.~117--144.

\bibitem{davis2019stochastic}
{\sc D.~Davis and D.~Drusvyatskiy}, {\em Stochastic model-based minimization of
  weakly convex functions}, SIAM Journal on Optimization, 29 (2019),
  pp.~207--239.

\bibitem{davis2019stochasticgeom}
{\sc D.~Davis, D.~Drusvyatskiy, and V.~Charisopoulos}, {\em Stochastic
  algorithms with geometric step decay converge linearly on sharp functions},
  arXiv preprint arXiv:1907.09547,  (2019).

\bibitem{davis2017nonsmooth}
{\sc D.~Davis, D.~Drusvyatskiy, and C.~Paquette}, {\em The nonsmooth landscape
  of phase retrieval}, arXiv preprint arXiv:1711.03247,  (2017).

\bibitem{doi:10.1137/17M1151031}
{\sc D.~Davis and B.~Grimmer}, {\em Proximally guided stochastic subgradient
  method for nonsmooth, nonconvex problems}, SIAM Journal on Optimization, 29
  (2019), pp.~1908--1930.

\bibitem{de2016new}
{\sc G.~de~Carvalho~Bento, J.~X. da~Cruz~Neto, and P.~R. Oliveira}, {\em A new
  approach to the proximal point method: convergence on general riemannian
  manifolds}, Journal of Optimization Theory and Applications, 168 (2016),
  pp.~743--755.

\bibitem{comp_DP}
{\sc D.~Drusvyatskiy and C.~Paquette}, {\em Efficiency of minimizing
  compositions of convex functions and smooth maps}, Mathematical Programming,
  178 (2019), pp.~503--558.

\bibitem{duchi2018stochastic}
{\sc J.~C. Duchi and F.~Ruan}, {\em Stochastic methods for composite and weakly
  convex optimization problems}, SIAM Journal on Optimization, 28 (2018),
  pp.~3229--3259.

\bibitem{facchinei2017ghost}
{\sc F.~Facchinei, V.~Kungurtsev, L.~Lampariello, and G.~Scutari}, {\em Ghost
  penalties in nonconvex constrained optimization: Diminishing stepsizes and
  iteration complexity}, arXiv preprint arXiv:1709.03384,  (2017).

\bibitem{federer1959curvature}
{\sc H.~Federer}, {\em Curvature measures}, Transactions of the American
  Mathematical Society, 93 (1959), pp.~418--491.

\bibitem{ferreira2002proximal}
{\sc O.~Ferreira and P.~Oliveira}, {\em Proximal point algorithm on riemannian
  manifolds}, Optimization, 51 (2002), pp.~257--270.

\bibitem{grohs2016varepsilon}
{\sc P.~Grohs and S.~Hosseini}, {\em $\varepsilon$-subgradient algorithms for
  locally lipschitz functions on riemannian manifolds}, Advances in
  Computational Mathematics, 42 (2016), pp.~333--360.

\bibitem{hosseini2018line}
{\sc S.~Hosseini, W.~Huang, and R.~Yousefpour}, {\em Line search algorithms for
  locally lipschitz functions on riemannian manifolds}, SIAM Journal on
  Optimization, 28 (2018), pp.~596--619.

\bibitem{hosseini2017riemannian}
{\sc S.~Hosseini and A.~Uschmajew}, {\em A {R}iemannian gradient sampling
  algorithm for nonsmooth optimization on manifolds}, SIAM Journal on
  Optimization, 27 (2017), pp.~173--189.

\bibitem{ioffe_survey}
{\sc A.~Ioffe}, {\em Metric regularity and subdifferential calculus}, Uspekhi
  Mat. Nauk, 55 (2000), pp.~103--162.

\bibitem{li2019nonsmooth}
{\sc X.~Li, S.~Chen, Z.~Deng, Q.~Qu, Z.~Zhu, and A.~M.~C. So}, {\em Nonsmooth
  optimization over stiefel manifold: Riemannian subgradient methods}, arXiv
  preprint arXiv:1911.05047,  (2019).

\bibitem{li2018nonconvex}
{\sc X.~Li, Z.~Zhu, A.~M.-C. So, and R.~Vidal}, {\em Nonconvex robust low-rank
  matrix recovery}, arXiv preprint arXiv:1809.09237,  (2018).

\bibitem{lin2019inexact}
{\sc Q.~Lin, R.~Ma, and Y.~Xu}, {\em Inexact proximal-point penalty methods for
  non-convex optimization with non-convex constraints}, arXiv preprint
  arXiv:1908.11518,  (2019).

\bibitem{ma2019proximally}
{\sc R.~Ma, Q.~Lin, and T.~Yang}, {\em Proximally constrained methods for
  weakly convex optimization with weakly convex constraints}, arXiv preprint
  arXiv:1908.01871,  (2019).

\bibitem{Mord_1}
{\sc B.~Mordukhovich}, {\em Variational Analysis and Generalized
  Differentiation I: Basic Theory}, Grundlehren der mathematischen
  Wissenschaften, Vol 330, Springer, Berlin, 2006.

\bibitem{nesterov2013gradient}
{\sc Y.~Nesterov}, {\em Gradient methods for minimizing composite functions},
  Mathematical Programming, 140 (2013), pp.~125--161.

\bibitem{NW}
{\sc J.~Nocedal and S.~Wright}, {\em Numerical optimization}, Springer Series
  in Operations Research and Financial Engineering, Springer, New York,
  second~ed., 2006.

\bibitem{penot_book}
{\sc J.-P. Penot}, {\em Calculus without derivatives}, vol.~266 of Graduate
  Texts in Mathematics, Springer, New York, 2013.

\bibitem{poliquin1996prox}
{\sc R.~Poliquin and R.~Rockafellar}, {\em Prox-regular functions in
  variational analysis}, Transactions of the American Mathematical Society, 348
  (1996), pp.~1805--1838.

\bibitem{RW98}
{\sc R.~Rockafellar and R.-B. Wets}, {\em Variational {A}nalysis}, Grundlehren
  der mathematischen Wissenschaften, Vol 317, Springer, Berlin, 1998.

\bibitem{teboulle2018simplified}
{\sc M.~Teboulle}, {\em A simplified view of first order methods for
  optimization}, Mathematical Programming, 170 (2018), pp.~67--96.

\bibitem{wang2017penalty}
{\sc X.~Wang, S.~Ma, and Y.-x. Yuan}, {\em Penalty methods with stochastic
  approximation for stochastic nonlinear programming}, Mathematics of
  Computation, 86 (2017), pp.~1793--1820.

\bibitem{10.1145/3362077.3362085}
{\sc T.~Yang}, {\em Advancing non-convex and constrained learning: Challenges
  and opportunities}, AI Matters, 5 (2019), p.~29–39.

\end{thebibliography}

%\appendix
\end{document}